\documentclass[11pt,reqno]{article}
\usepackage{amsmath}
\usepackage{amsthm}
\usepackage{amssymb}
\usepackage{amsfonts}
\usepackage{cases}
\usepackage{graphicx}
\usepackage{xcolor}
\usepackage[margin=0.9in]{geometry}
\usepackage[utf8]{inputenc}
\usepackage{verbatim}
\usepackage{caption,subcaption}
\usepackage{algorithm}
\usepackage{algpseudocode}
\usepackage[normalem]{ulem}
\makeatletter
\renewcommand{\fnum@algorithm}{\fname@algorithm}
\makeatother
\usepackage[colorlinks,citecolor=blue,urlcolor=blue]{hyperref}
\usepackage{bm}
\usepackage{enumerate}

\numberwithin{equation}{section}

\newtheorem{Remark}{Remark}[section]
\newtheorem{Theorem}{Theorem}[section]
\newtheorem{Lemma}{Lemma}[section]
\newtheorem{Proposition}{Proposition}[section]
\newtheorem{Corollary}{Corollary}[section]
\newtheorem{Assumption}{Assumption}[section]

\newcommand{\be}{\begin{equation}}
	\newcommand{\ee}{\end{equation}}
\newcommand{\bee}{\begin{equation*}}
	\newcommand{\eee}{\end{equation*}}
\newcommand{\bi}{\begin{itemize}}
	\newcommand{\ei}{\end{itemize}}
\DeclareMathOperator{\tr}{tr}
\DeclareMathOperator*{\argmax}{arg\,max}

\usepackage{algorithm}
\usepackage{algpseudocode}
\usepackage{soul}

\def \pit {\tilde{\pi}}

\def \E{\mathbb{E}}

\def \N{\mathbb{N}}
\def \P{\mathbb{P}}

\def \R{\mathbb{R}}

\def \Lc{{\mathcal L}}
\def \Cc{{\mathcal C}}
\def \Ac{{\mathcal A}}

\def \Pc{{\mathcal P}}
\def \Fc{{\mathcal F}}

\def \fz{\mathfrak{z}}
\def \Hc{\mathcal{H}}
\def \Leb{\operatorname{{{Leb}}}}
\def \poly{\mathfrak{p}}

\title{Convergence of Policy Iteration for Entropy-Regularized Stochastic Control Problems}
\author{Yu-Jui Huang\thanks{
		Department of Applied Mathematics, University of Colorado Boulder,  email: \texttt{yujui.huang@colorado.edu}. Partially supported by National Science Foundation (DMS-2109002).
	}
	\and 
	Zhenhua Wang\thanks{Department of Mathematics, Iowa State University, email: \texttt{zhenhuaw@iastate.edu}. 
	}
	\and 
	Zhou Zhou\thanks{School of Mathematics and Statistics, University of Sydney, email:
		\texttt{zhou.zhou@sydney.edu.au}.}
}
\begin{document}
	
	\maketitle
	
	\begin{abstract}
	For a general entropy-regularized stochastic control problem on an infinite horizon, we prove that a policy iteration algorithm (PIA) converges to an optimal relaxed control. Contrary to the standard stochastic control literature, classical H\"{o}lder estimates of value functions do not ensure the convergence of the PIA, due to the added entropy-regularizing term. To circumvent this, we carry out a delicate estimation by moving back and forth between appropriate H\"{o}lder and Sobolev spaces. This requires new Sobolev estimates designed specifically for the purpose of policy iteration and a nontrivial technique to contain the entropy growth. Ultimately, we obtain a uniform H\"older bound for the sequence of value functions generated by the PIA, thereby achieving the desired convergence result. Characterization of the optimal value function as the unique solution to an exploratory Hamilton--Jacobi--Bellman equation comes as a by-product. The PIA is numerically implemented in an example of optimal consumption. 
	\end{abstract}

\textbf{MSC (2010):} 
93E20,  
60H10, 
94A17 
\smallskip

\textbf{Keywords:} entropy regularization, policy iteration algorithm, exploratory Hamilton-Jacobi-Bellman equations, reinforcement learning.

\section{Introduction}\label{sec:intro}
Policy improvement refers to updating a current policy (i.e., control strategy) to a new one that improves the performance. In practice, one may repeat this procedure indefinitely (which forms a {\it policy iteration algorithm} (PIA)), with the hope that the recursive improvement will eventually lead to an optimal policy. 

For standard stochastic control problems in a diffusion model, 
the PIA generally lives up to one's hope. By classical estimates of partial differential equations (PDEs), Puterman \cite{Puterman81} shows that the optimal value is recovered under the PIA, for a specific control problem with a compact domain of time and space. Jacka and Mijatovi\'{c} \cite{jacka2017policy} put forth a long list of general assumptions under which the PIA converges to an optimal control, and a few examples that fulfill the assumptions are provided. Relying on backward stochastic differential equations, Kerimkulov et al.\ \cite{kerimkulov2020exponential} establish the convergence of the PIA to an 
optimal control with an appropriate  convergence rate, under the condition that only the drift coefficient of the state process is controlled.  
Overall, these developments share one underlying idea: for the desired convergence of the PIA, one needs a uniform bound for 
the value functions generated by the PIA as well as their derivatives---that is, a uniform bound in a H\"{o}lder space is in demand. Finding such a bound is made drastically more complicated by {\it entropy regularization}. 

Entropy regularization changes a standard stochastic control problem in two distinctive ways. At each time, it first randomizes the control action through a density function (or more generally, a probability measure) defined on the action space $U$; next, it adds the entropy of the density to the reward functional to be maximized. The ultimate goal is to find a density-valued stochastic process (also called a {\it relaxed control}) that maximizes the expected cumulative reward plus entropy. 

This class of problems has drawn substantial attention, due to its close ties to reinforcement learning (RL). In selecting a distribution under certain constraint (e.g., a performance criterion), the {\it principle of maximum entropy}, dating back to Jaynes \cite{Jaynes57-1, Jaynes57-2}, stipulates that the best choice is the distribution with largest possible entropy, among those that satisfy the constraint. As argued in \cite{Jaynes57-1}, maximum entropy amounts to least commitment to (or, dependence on) any additional assumption beyond the known constraint---put differently, in the words of Shannon \cite{Shannon48}, it preserves the highest level of information of the unknown. By applying the principle of maximum entropy to a dynamic discrete-time setting, Ziebart et al.\ \cite{Ziebart10} perform sequential optimization under a ``softmax'' criterion, which is equivalent to solving an entropy-regularized stochastic control problem. In the RL literature, it is now well-known that a ``softmax'' criterion (or entropy regularization) encourages exploration of the unknown environment, thereby preventing early settlement to suboptimal strategies; see Ziebart et al.\ \cite{Ziebart08}, Fox et al.\ \cite{Fox16}, Haarnoja et al.\ \cite{Haarnoja17}, and Haarnoja et al.\ \cite{Haarnoja18}, among many others. 

Lately, Wang et al.\ \cite{wang2019exploration} extended the above discrete-time RL setting to a continuous-time diffusion framework, which opens the door to mathematical analysis based on the Hamilton--Jacobi--Bellman (HJB) equation. This continuous-time analysis has been applied to entropy-regularized stochastic control problems in various settings. For linear--quadratic problems, an optimal relaxed control is derived explicitly in \cite{wang2019exploration} and applied to mean--variance portfolio selection in Wang and Zhou \cite{wang2020continuous}. Firoozi and Jaimungal \cite{MR4385154} and Guo et al.\ \cite{Guo22} generalize this single-agent setup to a mean field game and obtain the explicit form of a Nash equilibrium.
For general control problems beyond the linear--quadratic case, when the action space is finite, Reisinger and Zhang \cite{RZ21} study the regularity of the optimal value function and the related optimal relaxed control, as well as their stability with respect to model parameters and convergence in case of diminishing regularization. Tang et al.\ \cite{tang2021exploratory} obtain corresponding regularity and convergence results for the case of a general action space. 

Despite this vibrant development, the PIA in a diffusion framework remains largely unexplored under entropy regularization. To the best of our knowledge, 
the convergence of the PIA to an optimal relaxed control is established {\it only} in the linear--quadratic case. Specifically, the HJB analysis in \cite{wang2019exploration} suggests that the PIA be performed among Gibbs-form relaxed control (see e.g., \eqref{eq.gibbs.pi} below). As Gibbs-form distributions reduce to Gaussian densities in the linear--quadratic case, the PIA is highly tractable: as shown in \cite{wang2020continuous}, an optimal policy is reached within two iterations in the PIA. Even in a time-inconsistent case where one looks for an equilibrium policy (instead of a value-maximizing policy), Dai et al.\ \cite{dai2023learning} show that the linear--quadratic structure still ensures that equilibrium policies can take the form of Gaussian densities and the PIA converges to such a policy.\footnote{In such a PIA, an initial policy needs to be carefully selected based on the linear--quadratic structure to guarantee the convergence of algorithm. Also, the policy improvement property no longer holds, as the goal is no longer value maximization but attaining an equilibrium.} 

Let us stress the importance of establishing the PIA convergence for {\it general} entropy-regularized stochastic control problems. 
For a standard stochastic control problem, one can always find the optimal value function (and thus an optimal policy) by solving the associated HJB equation, even without considering the PIA at all. For an entropy-regularized stochastic control problem, as it provides a theoretic description of an RL problem (thanks to the exposition in \cite{wang2019exploration}), the study of the PIA has significant value. Indeed, in an actual RL scenario, as the underlying model is not perfectly known, the HJB equation cannot be readily solved. One instead has to rely on the PIA to gradually improve the policy at hand and ultimately approach an optimal one. 

In this paper, we take up a general entropy-regularized stochastic control problem beyond the linear--quadratic case. Our goal is to establish a general convergence result that the PIA yields an optimal relaxed control. As aforementioned, this boils down to finding a uniform bound in a H\"{o}lder space for the value functions $\{v^n\}_{n\in\N}$ generated by the PIA. 
Intriguingly, we find that classical Schauder estimates 
are not enough to provide the desired uniform bound, due to the added entropy term; 
see Remark~\ref{rem:detour necessary} for details. In response, we devise a delicate estimation procedure (i.e., the ``grand plan'' in Section~\ref{subsec:grand plan}) that goes back and forth between appropriate H\"{o}lder and Sobolev spaces. 

To actually implement this plan, we need preparations in two important aspects. First, we develop new Sobolev estimates specifically for the purpose of policy iteration (Section~\ref{subsec:Sobolev}). In particular, Lemma~\ref{lm.hnorm.estimate} carefully extends a classical interior Sobolev estimate to higher order, and further spells out explicitly how the estimate depends on the controlled drift of the underlying state process. This explicit dependence is normally overlooked in the standard PDE literature, but indispensable for our analysis of the PIA; see the discussion above Lemma~\ref{lm.hnorm.estimate} and Remark~\ref{rem:explicit beta} (i) for details. In addition, we identify suitable conditions to contain the growth of the entropy term (Section~\ref{subsec:entropy}). By assuming that the reward function and the drift coefficient of the state process are Lipschitz in the control variable  (Assumption~\ref{assume.u.lips}) and the action space $U\subset\R^\ell$ has a positive finite measure and fulfills a {\it uniform cone condition} (Assumption~\ref{assume.U.cone}), Lemma~\ref{lm.1} shows that any Gibbs-form distribution has polynomial growth, so that its entropy can grow only logarithmically; see Corollary~\ref{coro:H bdd}. 

Let us stress that containing the entropy growth is challenging in general. While the condition ``$U\subset\R^\ell$ has a positive finite measure'' naturally yields an upper bound for the entropy (see \eqref{eq.kl.lbound} below), the uniform cone condition of $U$ (Assumption~\ref{assume.U.cone}) is subtle: it relates any Gibbs-form distribution to an integral in $\R^\ell$ over a sector of a fixed size. Evaluating this integral (under the $\ell$-dimensional polar coordinates on a sector) gives a polynomial, as desired.

Based on all this, Proposition~\ref{thm.the.uniformbound} carries out the ``grand plan'' in detail, producing a desired uniform bound for $\{v^n\}_{n\in\N}$ in a H\"{o}lder space. This allows us to show $v^n\to V^*$ in the H\"{o}lder space and obtain an optimal relaxed control in terms of $V^*$; see Theorem~\ref{thm.verification.limit}, the main result of this paper. 
As a by-product of Theorem~\ref{thm.verification.limit}, $V^*$ is characterized as the unique classical solution to a so-called {\it exploratory} HJB equation (i.e., \eqref{eq.HJB.new} below) with sufficient regularity. This complements the well-posedness result in \cite{tang2021exploratory} for exploratory HJB equations, and particularly answers an open question raised in \cite[Section 6]{tang2021exploratory}; see Remark~\ref{rem:complements} for details. 

Most of our results require that only the drift coefficient of the state process is controlled. 
This is not very restrictive in terms of the literature of PIAs. Even for standard stochastic control problems (without entropy regularization), \cite{Puterman81} and \cite{kerimkulov2020exponential} assume outright that only the drift coefficient is controlled; \cite{jacka2017policy} does not assume this, but the explicit example provided there (see \cite[Remark 4 (i)]{jacka2017policy}) still allows only the drift to be controlled; \cite{kerimkulov2021modified} considers the general case with the diffusion coefficient also controlled, but the PIA only converges to a locally optimal control. 
Furthermore, from a technical point of view, our ``grand plan'' does not easily extend to the case of a controlled diffusion coefficient: as explained in Section~\ref{subsec:only drift controlled}, a crucial ``order reduction'' no longer holds in the more general setting.   


Finally, let us point out several recent studies related to the PIA, mostly for deterministic control without entropy regularization. For a state process modeled by a controlled ordinary differential equation (ODE), Lee and Sutton \cite{lee2021policy} propose two  versions of the PIA and analyze their performance for various RL problems in continuous time and space. When the state dynamics is additionally affine in the control action, Wallace and Si \cite{wallace2023continuous} provide a comprehensive review of four  implementations of the PIA; Lutter et al.\ \cite{lutter2021value}, on the other hand, propose a value iteration algorithm (which performs iterations on the value function directly). Under a general ODE dynamics, Tang et al.\ \cite{tang2023policy} prove a convergence rate for a semi-discrete version of the PIA, using viscosity solution techniques. 

The rest of the paper is organized as follows. Section~\ref{sec:notation} collects the notation in the paper. Section~\ref{sec:model} introduces the model, the PIA to consider, and two useful technical lemmas. Section \ref{sec:v.gibbs} investigates properties of the value function associated with a Gibbs-form relaxed control, including its regularity and PDE characterization. In Section~\ref{sec.pia.convergence}, we first observe that the PIA does improve a policy recursively. To show that the recursive improvement generates an optimal relaxed control, we propose a delicate estimation procedure in Section~\ref{subsec:grand plan}, whose implementation requires theoretic preparations in Sections~\ref{subsec:Sobolev} and \ref{subsec:entropy}. Section~\ref{subsec:optimality} performs the estimation procedure, which yields the desired result that the PIA yields an optimal relaxed control. Section~\ref{sec:example} presents an optimal consumption problem where our theoretic result can be applied. We numerically implement the PIA and show that it does lead to an optimal relaxed control in this example.

\subsection{Notation}\label{sec:notation}
Let $\N$ be the set of positive integers and define $\N_0 := \N\cup\{0\}$. Consider $\R_+:=(0,\infty)$. For any $a\in \R$, we denote by $\lfloor a\rfloor$ the largest integer $n$ such that $n\le a$. 

Fix any $m, n\in\N$. In $\R^m$, we denote by $|\cdot|$ the Euclidean norm and by {$B_\eta(x)$} the open ball centered at $x\in\R^m$ with radius $\eta>0$. For any $E\subseteq \R^m$, let $\Leb(E)$ be the Lebesgue measure of $E$ and $\dim E$ be the radius of the smallest closed ball that covers $E$ if such a ball exists. Given another $E'\subseteq \R^m$, we consider $\text{dist}(E,E') := \inf_{x\in E, y\in E'} |x-y|$ and write $E'\subset \subset E$ if $E'$ is compactly embedded in $E$. For any $x,y\in \R^m$ and $M\in \R^{m\times n}$, we denote by $x'$ and $M'$ the transpose of $x$ and $M$, respectively, and by $x\cdot y$ the inner product of $x$ and $y$. 

Given $f: \R^m\rightarrow \R$, let $D_x f$ and $D^2_x f$ denote the gradient and Hessian matrix of $f$, respectively. For a multi-index $a=(a_1,\ldots,a_m)$, with $a_i\in \N_0$ for all $i=1,\ldots,m$, we consider $|a|_{l_1} := \sum_{i=1}^m a_i$. For any $k\in\N_0$ and multi-index $a$ with $|a|_{l_1}=k$, $D^a_x f := \frac{\partial^{k}f}{\partial_{x_1}^{a_1}\ldots\partial_{x_m}^{a_m}}$ is called a $k^{th}$ order derivative of $f$, if it exists.  
Given a domain $E\subseteq \R^m$, $k\in\N_0$, and $0<\alpha\leq 1$, we define
\begin{align*}
	&[f]_{0,\alpha(E)}:=\|f\|_{L^\infty(E)}+\sup_{x,y\in E, x\neq y} \frac{|f(x)-f(y)|}{|x-y|^\alpha},\\
	[f]_{k(E)}:=&\underset{|a|_{l_1}=k}{\sum}\|D^a_x f\|_{L^\infty(E)}, \quad 	[f]_{k,\alpha(E)}:= \underset{|a|_{l_1}=k}{\sum}[D^a_x f]_{0,\alpha(E)}, \\
	{\|f\|_{\Cc^{k}(E)}} &:=\sum_{j=0,\ldots,k} [f]_{j(E)},\quad \|f\|_{\Cc^{k,\alpha}(E)}:=\sum_{j=0,\ldots,k} [f]_{j,\alpha(E)}.
\end{align*}
Let $\Cc^k(E)$ (resp. $\Cc^{k,\alpha}(E)$) be the set of $k$-times continuously differentiable $f:E\to \R$ with $\|f\|_{\Cc^{k}(E)}<\infty$ (resp.\ $\|f\|_{\Cc^{k,\alpha}(E)}<\infty$). 
For $E=\R^m$, we further consider ${\Cc^{k,\alpha}_{\text{unif}}(\R^m)}$, the set of $f\in\Cc^{k,\alpha}(\R^m)$ with 
$$
\|f\|_{\Cc^{k,\alpha}_{\text{unif}}(\R^m)} := \sup_{x\in \R^m} \|f\|_{\Cc^{k,\alpha}(B_1(x))}<\infty.
$$
By definition, $\Cc^{k,\alpha}_{\text{unif}}(\R^m)\subset \Cc^{k}(\R^m)$. 
Given $q\ge 1$, we also consider $W^{k,q}(E)$, the set of $f:E\to\R$ whose weak derivatives up to the $k^{th}$ order exist and lie in $L^q(E)$, i.e., $\|f\|_{W^{k,q}(E)}:=\sum_{|a|_{l_1}\le k} \| D^a_x f\|_{L^q(E)}<\infty$, where $D^a_x f$ now represents a weak derivative. For any vector-valued $g=(g_i)_{1\leq i\leq m}:E\to \R^m$ (resp.\ a matrix-valued $g=(g_{ij})_{1\le i\le m, 1\le j\le n}:E \rightarrow \R^{m\times n}$), we will write $g\in \Cc^{k}(E)$ whenever $g_i\in \Cc^{k}(E)$ for all $1\le i\le m$ (resp.\ $g_{ij}\in \Cc^k(E)$ for all $1\le i\le m$, $1\le j\le n$). Accordingly, we define $\|g\|_{\Cc^k(E)}=\sum_{i=1}^m \|g_i\|_{\Cc^k(E)}$ (resp.\ $\|g\|_{\Cc^k(E)}=\sum_{1\leq i\leq m, 1\leq j\leq n} \|g_{ij}\|_{\Cc^k(E)}$). Similar notation is used when we say a vector- or matrix-valued function belongs to $\Cc^{k,\alpha}(E)$, $\Cc^{k,\alpha}_{\text{unif}}(\R^m)$, or $W^{k,q}(E)$.


\section{The Setup and Preliminaries}\label{sec:model}
Fix $d, \bar d, \ell\in\N$. Consider a probability space $(\Omega,\Fc,\P)$ that supports a $\bar d$-dimensional Brownian motion $(W_t)_{t\geq 0}$, adapted to a filtration $\mathbb F=(\Fc_t)_{t\ge 0}$ with $\Fc_t\subseteq \Fc$ for all $t\ge 0$.  
Let $U\subset \R^\ell$ have a positive finite measure, i.e., $0<\Leb(U)<\infty$. Denote by $\Pc(U)$ the set of all probability density functions on $U$. A process $\pi=(\pi_t)_{t\ge 0}$ is called a {\it relaxed control} if it is $\mathbb F$-adapted with $\pi_t\in \Pc(U)$ for all $t\ge 0$.

Given Borel measurable functions $b=(b_1,\ldots,b_d): \R^d\times U\to \R^d$ and $\sigma= (\sigma_{ij})_{i, j}:\R^d\to \R^{d\times\bar d}$, an agent considers, for any relaxed control $\pi$ and $x\in\R^d$, the diffusion process
\be\label{eq.sde.new} 
dX^\pi_s=\bigg(\int_U b(X^\pi_s,u)\pi_s(u)du\bigg)  ds+ \sigma(X^\pi_s) dW_s\quad\hbox{for}\ s\ge 0,\quad X_0=x,
\ee
and the associated value function
\begin{equation}\label{V^pi}
	V^\pi(x):=\E_x\bigg[ \int_0^\infty  e^{-\rho s}\left( \int_U r(X_s^{\pi},u)\pi_s(u) du-\lambda \int_U \pi_s(u)\ln \pi_s(u)du \right) ds \bigg],
\end{equation}
where $\rho>0$ is the agent's discount rate, $r:\R^d\times U\to \R$ is a given reward function, assumed to be Borel measurable, and $\lambda>0$ is the weight the agent prescribes for the entropy-regularizing term ``$-\int_U \pi_s(u)\ln \pi_s(u)du$''.  
A relaxed control $\pi$ is said to be {\it admissible} if \eqref{eq.sde.new} admits a unique strong solution and $V^\pi(x)$ is finite for all $x\in\R^d$. We denote by $\Ac$ the set of all admissible relaxed controls. Furthermore, we call $\pi\in\Ac$ a {\it Markov} relaxed control if there exists a Borel measurable $\mu:\R^d\times U\to[0,\infty)$ such that $\pi_s(u) = \mu(X_s,u)$ for all $s\ge 0$ and $u\in U$, and denote by $\Ac_{M}$ the set of all Markov relaxed controls. Oftentimes, we will write ``$\pi\in\Ac_M$'' and ``$\mu\in\Ac_M$'' interchangeably.

The agent aims at achieving the optimal value  
\begin{equation}\label{V^*}
	V^*(x) := \sup_{\pi\in \Ac} V^{\pi}(x)
\end{equation}
by finding an optimal relaxed control $\pi^*\in\Ac$. In the search for $\pi^*$, {\it policy improvement} refers to updating an initial guess $\pi^0\in\Ac$ to a better relaxed control $\pi^1\in\Ac$, in the sense of $V^{\pi^1}\ge V^{\pi^0}$. One may repeat this  procedure indefinitely, hoping that $\pi^*$ can be uncovered in the limit.    


\subsection{Policy Iteration Algorithm}\label{subsec.pia}
In view of the derivation in \cite{wang2019exploration}, the HJB equation\footnote{The equation is also called an {\it exploratory} HJB equation, coined by Tang et al.\ \cite{tang2021exploratory}, to emphasize its connection to exploration in reinforcement learning.} associated with $V^*$ is
\be\label{eq.HJB.new}
\rho v(x)=\sup_{\varpi \in \Pc(U)}   \int_U \Big(b(x,u)\cdot D_x v(x)+r(x,u)-\lambda \ln\varpi(u) \Big)\varpi(u)du+\frac{1}{2}\tr((\sigma \sigma^T) D^2_x v)(x), 
\ee
and the resulting candidate optimizer $\pi^*\in\Ac_M$ is given in the Gibbs form 
\begin{equation}\label{pi^*}
	\pi^*(x,u) := \Gamma(x,D_x V^*(x),u),
\end{equation}
where $\Gamma:\R^d\times\R^d\times U\to \R$ is defined by 
\begin{align}\label{eq.gibbs.pi}  
	\Gamma(x,y,u) &:= \argmax_{\varpi\in \Pc(U)}\int_U \big(b(x,u)\cdot y+r(x,u)-\lambda \ln(\varpi(u)) \big)\varpi(u) du \notag\\
	&= \frac{\exp(\frac{1}{\lambda} [b(x,u)\cdot y+r(x,u)])}{\int_U \exp(\frac{1}{\lambda} [b(x,\tilde u)\cdot y+r(x,\tilde u)]) d \tilde u}.
\end{align}
In view of this, we propose the following {\it policy iteration algorithm} (PIA): 
\begin{itemize}
	\item Take an arbitrary $v^0\in \Cc^{1,1}_{\text{unif}}(\R^d)$.  
	For each $n\in\N$, introduce 
	\begin{equation}\label{PIA}
		\pi^n(u,x) := \Gamma(x, D_x v^{n-1}(x),u)\in\Ac_M\quad \hbox{and}\quad v^n := V^{\pi^n}. 
	\end{equation}
\end{itemize}
The goal of this paper is to show that this algorithm is well-defined and fulfills two key properties: 
\begin{itemize}
\item[(i)] It produces improved relaxed controls recursively, i.e. $v^{n}\ge v^{n-1}$ for all $n\in\N$ (Proposition~\ref{prop.policy.improvement});
\item[(ii)] It leads to the optimal value function $V^*$ and yields an optimal relaxed control, i.e., $v^n\to V^*$ and $\pi^*$ given in \eqref{pi^*} is indeed an optimizer for \eqref{V^*} (Theorem~\ref{thm.verification.limit}). 
\end{itemize}


\subsection{Two Fundamental Lemmas}
Let us present two lemmas that will be useful in Sections~\ref{sec:v.gibbs} and \ref{sec.pia.convergence} below. First, we observe that $V^\pi$ has a {\it uniform} upper bound (independent of $\pi\in\mathcal A$), as long as the reward function $r$ is bounded.  

\begin{Lemma}\label{lem:V^pi bdd}
	For any $\pi\in\Ac$,  
	$V^\pi(x) \leq \frac1\rho \big(\|r\|_{\Cc^0(\R^d)}+\lambda|\ln(\Leb (U))|\big)$ for all $x\in\R^d$.
\end{Lemma}

\begin{proof}
	Consider the uniform density $\nu\in\Pc(U)$ given by $\nu(u):=1/\Leb(U)$ for all $u\in U$. For any $f\in \Pc(U)$, we compute the Kullback--Leibler divergence
	\begin{equation}\label{eq.kl.lbound}
		0\leq D_{KL}(f\|\nu) :=\int_U f\ln\left(\frac{f}{\nu}\right)du=\int_U f\ln f du+\ln (\Leb(U)),
	\end{equation}
	which gives $-\int_U f\ln f du\leq \ln (\Leb(U))$. In view of \eqref{V^pi} and this inequality,
	\begin{align*}
		V^\pi(x)\leq \int_0^\infty e^{-\rho s}(\|r\|_{\Cc^0(\R^d)}+\lambda \ln \Leb(U))ds=\frac1\rho \big({\|r\|_{\Cc^0(\R^d)}+\lambda \ln (\Leb(U))}\big)\ \ \forall x\in\R^d,
	\end{align*}
	which holds  for any $\pi\in \Ac$.
\end{proof}

Next, we explore regularity of certain functions composed of $r(x,u)$, $b(x,u)$, and $\Gamma(x,y,u)$.  
Assume that there exists $k\in\N_0$ such that 
\be\label{eq.parameters.m}  
\Lambda_k := \sup_{u\in U}\left\{\|r(\cdot,u)\|_{\Cc^k({\R^d})}+ \|b(\cdot,u)\|_{\Cc^k({\R^d})}+\|\sigma\|_{\Cc^k({\R^d})}\right\}
\ee
is a finite number. In addition, for any generic function $f=f(x,y,u):\R^d\times\R^d\times U\rightarrow \R$, we define 
\be\label{eq.hat.notation}
\hat{f}(x,y):=\int_U f(x,y,u)\Gamma(x,y,u)du\quad \forall (x,y)\in\R^d\times \R^d.
\ee 
We also consider $\mathcal{H}:\R^d\times\R^d\times U\rightarrow \R$ defined by 
\be\label{eq.h} 
\begin{aligned}
	\Hc(x,y, u) &:=\lambda \ln \Gamma(x,y,u)\\
	&= b(x,u)\cdot  y+r(x,u)-\lambda \ln\left(\int_U \exp\left(\frac{1}{\lambda} [b(x, \tilde u)\cdot  y+r(x,\tilde u)]\right) d \tilde u\right).
\end{aligned}
\ee 

\begin{Lemma}\label{lemma.lips.xy}	
	Fix $k\in\N_0$ and $p:\R^d\to \R^d$. Set $f(x,y,u) := r(x,u)-\Hc(x,y,u)$ (or $f(x,u) := b(x,u)$).
\begin{enumerate}[\normalfont(i)]
		\item If $\Lambda_{k}<\infty$ in \eqref{eq.parameters.m} and $p\in \Cc^{k}(\R^d)$, then 
		for any multi-index $a=(a_1,..,a_d)$ with $|a|_{l_1}=k$, 
		\be\label{eq.lips.xy}
		\left|D^a_x \hat{f}(x,p(x))\right|\leq C_k \left(1+|p(x)|^{(k+1)2^k}\right)\bigg(1+\sum_{|c|_{l_1}=0,\ldots,k}|D^c_x p(x)|\bigg)^k\quad \forall x\in\R^d,
		\ee
		where $C_k>0$ is a constant depending on only $\Lambda_k$, $k$, $\lambda$, and $d$. Moreover, $\hat f(\cdot, p(\cdot))\in \Cc^{k}(\R^d)$.
		\item If $\Lambda_{k+1}<\infty$ in \eqref{eq.parameters.m} and $p\in\Cc^{k,\alpha}_{\text{unif}}(\R^d)$ with $0<\alpha\leq 1$, then $\hat f(\cdot, p(\cdot))\in \Cc^{k,\alpha}_{\text{unif}}(\R^d)$. 
\end{enumerate}
\end{Lemma}
The proof of Lemma~\ref{lemma.lips.xy} is relegated to Appendix \ref{sec:appendix}. 


\section{Value Functions for Gibbs-Form Relaxed Control}\label{sec:v.gibbs}
In view of \eqref{PIA}, every $\pi^n$ generated in the PIA is a Gibbs-form relaxed control given by 
\begin{equation}\label{pi=Gamma}
	\pi(x,u) = \Gamma(x,p(x),u)\quad \hbox{for some}\ p:\R^d\to \R^d. 
\end{equation}
In this section, we will investigate properties of the value functions $V^\pi$ associated with $\pi$ in the form of \eqref{pi=Gamma}. The first question is whether \eqref{pi=Gamma} is admissible. The next result provides a sufficient condition, which also ensures that $V^\pi$ is well-defined (i.e., finitely-valued) and continuous. 

\begin{Proposition}\label{prop.sde}
	Let $\Lambda_1<\infty$ in \eqref{eq.parameters.m}. For any $p:\R^d\rightarrow \R^d$ of the class $\mathcal{C}^{0,1}_{\text{unif}}({\R^d})$, $\pi$ given by \eqref{pi=Gamma} is an admissible relaxed control, whence $\pi\in\Ac_M$. Moreover, $V^\pi$ is bounded, with
	\[
	|V^\pi(x)|\le 3\Lambda_0(1+\|p\|_{\Cc^0(\R^d)})+ \lambda|\ln(\Leb(U))|\quad \forall x\in\R^d, 
	\]
	and also continuous on $\R^d$.
\end{Proposition}

\begin{proof}	
	By $\Lambda_1<\infty$ in \eqref{eq.parameters.m}, $r$, $b$, and $\sigma$ are bounded on $\R^d\times U$ and globally Lipschitz in $x$, uniformly in $u$. In view of \eqref{eq.gibbs.pi}, the boundedness of $r$, $b$, and $\sigma$, along with $0<\Leb(U)<\infty$, implies that $\Gamma(x,p(x),u)$ is well-defined and $u\mapsto \Gamma(x,p(x),u)$ belongs to $\Pc(U)$. Hence, for $\pi(x,u)=\Gamma(x,p(x),u)$ to be an admissible relaxed control, it remains to show that there is a unique strong solution to \eqref{eq.sde.new} and $|V^\pi(x)|<\infty$ for all $x\in\R^d$. Note from \eqref{eq.hat.notation} that \eqref{eq.sde.new} can be written as
	\be\label{eq.sde.new.z}
	dX^\pi_t=\hat{b}(X^\pi_t, p(X^\pi_t)) dt+ \sigma(X^\pi_t)dW_t,\quad X^{\pi}_0=x.
	\ee
	By Lemma~\ref{lemma.lips.xy} (ii) (with $k=0$), $\hat{b}(\cdot,p(\cdot))\in \Cc^{0,1}_{\text{unif}}(\R^d)$, i.e., $\hat{b}(\cdot,p(\cdot))$ is uniformly locally Lipschitz. This readily implies that $\hat{b}(\cdot,p(\cdot))$ is globally Lipschitz. Indeed, for any $y,z\in\R^d$, we can take $\{y_i\}_{i=0}^n$ on the line segment joining $y$ and $z$ such that $y_0=y$, $y_n=z$, $|y_i-y_{i-1}|\le 1$ for all $i=1,\ldots,n$, and $\sum_{i=1}^n |y_i-y_{i-1}|=|y-z|$. Then, by taking $M=\|\hat{b}(\cdot,p(\cdot))\|_{\Cc^{0,1}_{\text{unif}}(\R^d)}$, we get
	\[
	|\hat{b}(y,p(y)) - \hat{b}(z,p(z))| \le \sum_{i=1,...,n} |\hat{b}(y,p(y_i)) - \hat{b}(z,p(y_{i-1}))| \le M \sum_{i=1,...,n} |y_i - y_{i-1}| =M |y-z|,
	\]
	i.e., $\hat{b}(\cdot,p(\cdot))$ is globally Lipschitz. This, together with $\sigma$ being globally Lipschitz, ensures the existence of a unique strong solution to \eqref{eq.sde.new}.
	
	To prove the boundedness of $V^\pi$, consider the space $C([0,\infty);\R^d)$ of continuous paths $\omega:[0,\infty)\to\R^d$,  equipped with the metric 
$
	\operatorname{d}(\omega^1, \omega^2)=\sum_{n\in\N} \frac{1}{2^n} \sup_{s\in[0,n]} |\omega^1_s-\omega^2_s|\ \hbox{for}\ \omega^1,\omega^2\in C([0,\infty);\R^d).
$
	Recall \eqref{eq.hat.notation} and \eqref{eq.h}. Let us define $\Delta: (C([0,\infty);\R^d), \operatorname{d})\rightarrow (\R, |\cdot|)$ by
	\begin{equation}\label{DELTA}
		\Delta(\omega):=\int_0^\infty e^{-\rho s}\Big( \hat{r}(\omega_s,p(\omega_s))-\hat{\Hc}(\omega_s, p(\omega_s))\Big) ds 
	\end{equation}
	and observe from \eqref{V^pi} that 
	\begin{equation}\label{V^pi=E}
		V^\pi(x) = \E[\Delta(X^\pi_\cdot)].
	\end{equation} 
	In view of $\Hc$ in \eqref{eq.h} and $\Lambda_0$ in \eqref{eq.parameters.m}, we have
	\be\label{estimate in step 3} 
	\begin{aligned}
		{\Hc}(x,p(x),u)| 
		\leq & \Lambda_0 \left(1+\|p\|_{\Cc^0(\R^d)} \right)+\lambda\left|\ln\left(\exp\left(\frac{1}{\lambda}\Lambda_0 \left(1+\|p\|_{\Cc^0(\R^d)} \right)\right){\Leb}(U)\right)\right| \\
		 \leq & 2\Lambda_0 \left( 1+\|p\|_{\Cc^0(\R^d)} \right)+ \lambda |\ln(\Leb(U))|<\infty,
	\end{aligned}
\ee
	so that $\hat{\Hc}(x,p(x)) = \int_U \Hc(x,p(x),u)\Gamma(x,p(x),u)du \le 2\Lambda_0(1+\|p\|_{\Cc^0(\R^d)})+ \lambda |\ln(\Leb(U))|$. Similarly, as $|r(x,u)|\le \Lambda_0$, we have $|\hat{r}(x,p(x))|\leq \Lambda_0$. It follows that
	$
	\|\hat{r}(\cdot,p(\cdot))-\hat{\Hc}(\cdot,p(\cdot))\|_{\Cc^0(\R^d)}\leq K,
	$
	with $K :=3\Lambda_0(1+\|p\|_{\Cc^0(\R^d)})+ \lambda|\ln(\Leb(U))|<\infty.$ 
	By \eqref{V^pi=E}, this implies $|V^\pi(x)|\le K/\rho <\infty$. 
	
	To prove that $V^\pi$ is continuous, take $x$ and $\{x_n\}_{n\in\N}$ in $\R^d$ with $x_n\to x$ and denote by $X^y$ the strong solution to \eqref{eq.sde.new} with $X_0=y$, for $y\in\{x, x_1, x_2,\ldots\}$. For any $\varepsilon>0$, take $T>0$ such that 
	\be\label{eq.T.tail}
	\int_T^\infty e^{-\rho s}\|\hat{r}(\cdot,p(\cdot))-\hat{\Hc}(\cdot,p(\cdot))\|_{\Cc^0(\R^d)}ds\leq \frac{K}{\rho} e^{-\rho T}<\varepsilon.
	\ee
	Since $\hat{r}(\cdot, p(\cdot))-\hat{\Hc}(\cdot,p(\cdot))$ is Lipschitz (by a similar use of Lemma~\ref{lemma.lips.xy} as indicated below \eqref{eq.sde.new.z}), there exists $L>0$ such that for any $\omega^1,\omega^2\in C([0,\infty);\R^d)$,
	\be\label{eq.w12,lip}  
	\begin{aligned}
	\left|\left(\hat{r}(\omega^1_s,p(\omega^1_s))-\hat{\Hc}(\omega^1_s, p(\omega^1_s))\right)- \left(\hat{r}(\omega^2_s,p(\omega^2_s))-\hat{\Hc}(\omega^2_s, y(\omega^2_s))\right)\right| 
\leq  L |\omega^1_s-\omega^2_s|\quad \forall s\geq 0. 
	\end{aligned}
\ee
	Also, by the Lipschitz continuity of $\hat{b}(\cdot,p(\cdot))$ and $\sigma(\cdot)$, a classical use of Gronwall's inequality yields
	\begin{equation}\label{Gronwall's}
		\E\bigg[\sup_{0\leq s\leq T} |X^{x_n}_s-X^x_s|\bigg] \leq C_T |x_n-x|,
	\end{equation}
	where $C_T>0$ is a constant depending on $T$. We then deduce from \eqref{DELTA}, \eqref{V^pi=E} and \eqref{eq.w12,lip} that
	\begin{align*}
		|V^\pi(x_n)-V^\pi(x)| &\leq  \E \left[|\Delta(X^{x_n}_\cdot)-\Delta(X^x_\cdot)|\right] \\
		 &\leq \E\bigg[ \int_0^T e^{-\rho s} L|X^{x_n}_s-X^x_s| ds
		+\int_T^\infty e^{-\rho s} 2\|\hat{r}(\cdot,p(\cdot))-\hat{\Hc}(\cdot,p(\cdot))\|_{\Cc^0(\R^d)} ds\bigg]\\
		&\leq  {L C_T}(1- e^{-\rho T}) |x_n-x|/\rho +2 \varepsilon,  
	\end{align*}
	where the last inequality follows from \eqref{Gronwall's} and \eqref{eq.T.tail}. This implies $\limsup_{n\to\infty} |V^\pi(x_n)-V^\pi(x)|\le 2\varepsilon$. As $\varepsilon>0$ is arbitrarily picked, we conclude $ \lim_{n\to\infty}|V^\pi(x_n)-V^\pi(x)|=0$, as desired. 
\end{proof}

To further investigate the regularity of $V^\pi$, we will invoke classical regularity results for elliptic equations. On an open connected domain $E\subseteq \R^d$, given $\gamma>0$, $\beta: E \rightarrow \R^d$, and $\delta: E \rightarrow \R^{d\times \bar d}$, we consider the elliptic operator 
\begin{equation}\label{L}
	L^{(\gamma,\beta,\delta)} w:= -\gamma w + \beta\cdot D_x w +\frac{1}{2}\tr(\delta\delta' D^2_x w ),\quad \forall w\in C^{2}(E).
\end{equation}
The next lemma summarizes results from Theorem 6.13, Theorem 6.17, Problem 6.1 in \cite{MR1814364}. 

\begin{Lemma}\label{lemma.pde.solution}
	Fix $k\in\N_0$ and $\alpha\in (0,1)$. Let $E\subset\R^d$ be a bounded connected open domain satisfying the exterior sphere condition.\footnote{$E$ is said to satisfy the {\it exterior sphere condition} if for each $x\in \partial E$, there exist $y\in\R^d$ and $r>0$ such that $B_r(y)\subset \R^d\setminus \overline{E}$ and $\overline{B_r(y)}\cap \overline{E}=\{x\}$. In this paper, since all the domains we consider will be either the whole space $\R^d$ or small balls in $\R^d$, this condition always holds.} Suppose that $\beta: E \rightarrow \R^d$ and $\delta: E \rightarrow \R^{d\times \bar{d}}$ satisfy
	\begin{align}
		M_k := \|\beta\|_{\Cc^{k,\alpha}(E)}+\|\delta\|_{\Cc^{k,\alpha}(E)}  <\infty \label{eq.lm.bsigma}
	\end{align}
	and there exists $K>0$ such that 
	\begin{equation}\label{tr bdd}
		\tr(\delta(x)\delta(x)' y' y)\geq K|y|^2\quad  \forall x\in E\  \hbox{and}\  y\in \R^d. 
	\end{equation}
	Then, for any $\gamma>0$, $f\in \Cc^{k,\alpha}(E)$, and $\phi:\overline E\to \R$ that is continuous,  the Dirichlet problem
	\begin{align}
		L^{(\gamma,\beta,\delta)}w&=f\quad \hbox{in}\ \ E,\label{L=f}\\
		w&=\phi\quad  \hbox{on}\ \ \partial E,
	\end{align}
	with $L^{(\gamma,\beta,\delta)}$ as in \eqref{L}, admits a unique solution $w\in \mathcal{C}^{k+2,\alpha}(E)$. Moreover, for any $E'\subset\subset E$,
	\be\label{eq.u.ck}  
	\|w\|_{\Cc^{k+2,\alpha}(E')}\leq C_k (\|w\|_{\Cc^0(E)}+\|f\|_{\Cc^{k,\alpha}(E)}),
	\ee
	for some $C_k>0$ that depends on $k$, $M_{k}$, $\operatorname{dist}(E', \partial E)$, $\dim(E)$, $\gamma$, $\alpha$, $K$, $d$, and $\bar d$.
\end{Lemma}

For any $\pi\in\Ac_M$ given by \eqref{pi=Gamma}, we recall \eqref{eq.hat.notation} and define the elliptic operator 
\be\label{eq.elliptic.gibbs} 
\Lc^\pi w := L^{(\rho,\hat b(\cdot,p(\cdot)),\sigma)} w = -\rho w +\hat{b}(\cdot,p(\cdot))\cdot D_x w+ \frac{1}{2}\tr(\sigma\sigma' D^2_x w),\quad \forall w\in C^2(\R^d).
\ee

\begin{Proposition}\label{prop.Vpi}
	Fix $\alpha\in(0,1)$ and let $\Lambda_{k+1}<\infty$ in \eqref{eq.parameters.m} for some $k\in\N_0$. Suppose that $\sigma$ fulfills the ellipticity condition, i.e., there exists $\eta_0>0$ such that
	\be\label{eq.assume.ellip}  
	\tr(\sigma(x) \sigma(x)'\xi'\xi)\ge \eta_0|\xi|^2\quad \forall \xi,x\in \R^d. 
	\ee
	For any $p:\R^d\rightarrow \R^d$ of the class $\Cc^{0,1}_{\text{unif}}(\R^d)\cap\Cc^{k,\alpha}_{\text{unif}}(\R^d)$, consider $\pi$ given by \eqref{pi=Gamma}. Then, we have  $\pi\in \Ac_M$,  $V^\pi\in\Cc^{k+2,\alpha}_{\text{unif}}(\R^d)$, and that $V^\pi$ satisfies the elliptic equation
	\be\label{eq.HJB.new.z}
	\begin{aligned}
		\Lc^\pi V^\pi(x)+\hat{r}(x,p(x))-\hat{\Hc}(x,p(x))=0\quad \forall x\in \R^d.
	\end{aligned}
	\ee
\end{Proposition}

\begin{proof}
	By Proposition~\ref{prop.sde}, $\pi\in\Ac_M$ and $V^\pi$ is bounded continuous on $\R^d$. 
	For a bounded connected open domain $E\subset \R^d$ satisfying the exterior sphere condition, 
	consider the Dirichlet problem
	\begin{align}\label{Dirichlet L^pi}
		L^{(\rho,\hat b(\cdot,p(\cdot)),\sigma)} w +\hat{r}(\cdot,p(\cdot))-\hat{\Hc}(\cdot,p(\cdot))&=0\ \ \hbox{in}\ E,\qquad w=V^\pi\ \ \hbox{on}\ \partial E. 
	\end{align}
	In view of \eqref{eq.parameters.m}, $\Lambda_{k+1}<\infty$ directly gives 
	\begin{equation}\label{sigma C^k,alpha}
		\sigma\in \Cc^{k+1}(\R^d)\subset \Cc^{k,\alpha}(\R^d)\subset \Cc^{k,\alpha}_{\text{unif}}(\R^d). 
	\end{equation}
	Also, with $\Lambda_{k+1}<\infty$ and $p\in \Cc^{k,\alpha}_{\text{unif}}(\R^d)$, Lemma \ref{lemma.lips.xy} (ii) entails 
	\begin{equation}\label{b C^k,alpha}
		\hat b(\cdot,p(\cdot)),\ \hat{r}(\cdot,p(\cdot))-\hat{\Hc}(\cdot,p(\cdot))\in\Cc^{k,\alpha}_{\text{unif}}(\R^d). 
	\end{equation}
	As $E$ is a bounded domain, \eqref{sigma C^k,alpha} and \eqref{b C^k,alpha} readily imply that $\sigma$, $\hat b(\cdot,p(\cdot))$, and $\hat{r}(\cdot,p(\cdot))-\hat{\Hc}(\cdot,p(\cdot))$ all belong to $\Cc^{k,\alpha}(E)$. Such regularity, along with \eqref{eq.assume.ellip} and the continuity of $V^\pi$ on $\overline E$, allows us to apply Lemma \ref{lemma.pde.solution} to conclude that there exists a unique solution $w\in \Cc^{k+2,\alpha}(E)$ to \eqref{Dirichlet L^pi}. Now, we claim that $w=V^\pi$ on $E$. For any $x\in E$, let $X^\pi$ denote the unique strong solution to \eqref{eq.sde.new} with $X_0=x$ and consider the stopping time $\tau_E:=\inf\{t\ge 0: X^{\pi}_t\notin E\}$. As \eqref{eq.assume.ellip} ensures $\E[\tau_E]<\infty$ (see e.g., \cite[Lemma 5.7.4 and Remark 5.7.5]{KS-book-91}), \cite[Proposition 5.7.2]{KS-book-91} asserts that
	\[
	w(x) = \E\bigg[\int_0^{\tau_E} e^{-\rho t}\left(\hat{r}(X^\pi_t, p(X^\pi_t))-\hat{\Hc}(X^\pi_t, p(X^\pi_t))\right)dt + e^{-\rho \tau_E} V^\pi(X^\pi_{\tau})\bigg]=V^\pi(x),
	\]
	where the last equality follows from \eqref{V^pi=E} and the strong Markov property of $X^\pi$. Hence, $w=V^\pi$ on $E$. As the domain $E$ is arbitrarily picked, this implies that $V^\pi$ satisfies \eqref{eq.HJB.new.z}. 
	
	It remains to prove $V^\pi\in \Cc^{k+2,\alpha}_{\text{unif}}(\R^d)$, i.e., $\sup_{x\in \R^d}\|V^\pi\|_{\Cc^{k+2,\alpha}(B_{1}(x))}<\infty$. By taking $E'=B_{1}(x)$ and $E = B_2(x)$ in \eqref{eq.u.ck} and considering the supremum over all $x\in\R^d$, we get
	\begin{align}
		\sup_{x\in \R^d}\|V^\pi\|_{C^{k+2,\alpha}(B_{1}(x))} &\leq  \sup_{x\in \R^d} C_k(x) \left(\|V^\pi\|_{\Cc^0(B_2(x))}+\|\hat{r}(\cdot,p(\cdot))-\hat{\Hc}(\cdot,p(\cdot))\|_{\Cc^{k,\alpha}(B_2(x))}\right)\notag\\
		&\leq \bigg(\sup_{x\in \R^d} C_k(x)\bigg) \bigg(\|V^\pi\|_{\Cc^0(\R^d)}+2 \|\hat{r}(\cdot,p(\cdot))-\hat{\Hc}(\cdot,p(\cdot))\|_{\Cc^{k,\alpha}_{\text{unif}}(\R^d)}\bigg),\label{finite term}
	\end{align}
	where $C_k(x)>0$ is a nondecreasing function of $\|\hat b(\cdot,p(\cdot))\|_{\Cc^{k,\alpha}(B_{2}(x))}$ and $\|\sigma\|_{\Cc^{k,\alpha}(B_{2}(x))}$. Recall from \eqref{sigma C^k,alpha}-\eqref{b C^k,alpha} that $\sigma, \hat b(\cdot,p(\cdot))\in\Cc^{k,\alpha}_{\text{unif}}(\R^d)$, so that 
	$
	\sup_{x\in\R^d} \|f\|_{\Cc^{k,\alpha}(B_{2}(x))} \le 2 \|f\|_{\Cc^{k,\alpha}_{\text{unif}}(\R^d)}<\infty,\ \hbox{for}\ f= \sigma,\ \hat b(\cdot,p(\cdot)).  
	$ 
	Hence, $\sup_{x\in \R^d} C_k(x)<\infty$. This, along with $V^\pi$ being bounded and $\hat{r}(\cdot,p(\cdot))-\hat{\Hc}(\cdot,p(\cdot))\in \Cc^{k,\alpha}_{\text{unif}}(\R^d)$ (by \eqref{b C^k,alpha}), ensures that the right-hand side of \eqref{finite term} is finite, as desired. 
\end{proof}

\begin{Remark}
	Proposition~\ref{prop.Vpi} requires both ``$p\in \Cc^{0,1}_{\text{unif}}(\R^d)$'' and ``$p\in \Cc^{k,\alpha}_{\text{unif}}(\R^d)$''. 
	The former ensures that $\pi$ given by \eqref{pi=Gamma} is an admissible relaxed control (Proposition~\ref{prop.sde}); the latter allows Lemma~\ref{lemma.lips.xy} to be applied, so that classical results for elliptic equation (i.e., Lemma~\ref{lemma.pde.solution}) can be imported for the study of $V^\pi$. In reality, only one of the two requirements takes effect (as it covers the other): when $k=0$ (resp.\ $k\ge1$), ``$p\in \Cc^{0,1}_{\text{unif}}(\R^d)$'' (resp.\ ``$p\in \Cc^{1,\alpha}_{\text{unif}}(\R^d)$'') takes effect. 
\end{Remark}

When put in the context of policy iteration, Proposition~\ref{prop.Vpi} reveals how the regularity of $v^n$ in \eqref{PIA} evolves. 
Simply put, the PIA upgrades the regularity of $v^n$ up to the regularity of $b$, $\sigma,$ and $r$.

\begin{Corollary}\label{prop.vn.bound}
	Let $\Lambda_{\bar k}<\infty$ in \eqref{eq.parameters.m} for some $\bar k\in \N$ and assume \eqref{eq.assume.ellip}. 
	Then, in the PIA \eqref{PIA}, $\pi^n\in\Ac_M$ for all $n\in\N$ and the sequence $\{v^n\}_{n\geq 0}$ satisfy the following for all $0<\alpha<1$:
	\begin{align*}
		v^n\in \Cc^{n+1,\alpha}_{\text{unif}}(\R^d),\quad \text{for}\ \ 1\leq n\leq \bar k; \qquad v^n\in \Cc^{\bar k+1,\alpha}_{\text{unif}}(\R^d),\quad \text{for}\ \ n\ge\bar k+1.  
	\end{align*} 
	Moreover, for each $n\in\N$, 
	\begin{equation}\label{v^n-1 to v^n}
		\Lc^{\pi^n} v^n(x)+\hat{r}(x, D_x v^{n-1}(x))-\hat{\Hc}(x, D_x v^{n-1}(x))=0\quad \forall x\in\R^d.
	\end{equation}
\end{Corollary}

\begin{proof}
	As $v^0\in \Cc^{1,1}_{\text{unif}}(\R^d)$, we have $D_x v^0\in\Cc^{0,1}_{\text{unif}}(\R^d)$. By taking $p = D_x v^0$ and $k=0$ in Proposition \ref{prop.Vpi}, we have $\pi^1\in\Ac_M$, $v^1\in \Cc^{2,\alpha}_{\text{unif}}(\R^d)$, and that \eqref{v^n-1 to v^n} holds for $n=1$. Similarly, by taking $p = D_x v^1\in\Cc^{1,\alpha}_{\text{unif}}(\R^d)$ and $k=1$ in Proposition \ref{prop.Vpi}, we have $v^2\in\Ac_M$, $v^2\in \Cc^{3,\alpha}_{\text{unif}}(\R^d)$, and that \eqref{v^n-1 to v^n} holds for $n=2$. Under $\Lambda_{\bar k}<\infty$, this procedure (i.e., taking $p = D_x v^{n-1}$ and $k=n-1$ in Proposition \ref{prop.Vpi}) can be carried out for all $n=1,2,\ldots,\bar k$, which shows that $v^n\in\Ac_M$, $v^n\in \Cc^{n+1,\alpha}_{\text{unif}}(\R^d)$ and \eqref{v^n-1 to v^n} holds for all $n=1,2,\ldots,\bar k$. Now, by taking $p = D_x v^{\bar k}\in \Cc^{\bar k,\alpha}_{\text{unif}}(\R^d)\subset \Cc^{\bar k-1,\alpha}_{\text{unif}}(\R^d)$ and $k=\bar k-1$ in Proposition \ref{prop.Vpi}, we have $\pi^{\bar k+1}\in\Ac_M$,  $v^{\bar k+1}\in \Cc^{\bar k+1,\alpha}_{\text{unif}}(\R^d)$, and that \eqref{v^n-1 to v^n} holds for $n=\bar k+1$. By repeating the same procedure (i.e., taking $p = D_x v^{n-1}$ and $k=\bar k-1$ in Proposition \ref{prop.Vpi}) for all $n> \bar k+1$, we see that $v^n\in\Ac_M$, $v^{n}\in \Cc^{\bar k+1,\alpha}_{\text{unif}}(\R^d)$, and \eqref{v^n-1 to v^n} holds for all $n>\bar k+1$. 
\end{proof}


\section{Policy Improvement and the Convergence of PIA}\label{sec.pia.convergence} 
The goal of this section is to establish the convergence of the PIA \eqref{PIA} to the optimal value $V^*$ in \eqref{V^*}. First, we will show that the algorithm does improve a policy recursively, i.e., $\{v^n\}_{n\in\N}$ in \eqref{PIA} satisfies $v^{n+1}\ge v^n$, and the limiting value $v^*$ is well-defined; see Proposition~\ref{prop.policy.improvement} and Corollary~\ref{coro.PIA.works}. Proving $v^*=V^*$ is much more challenging, due to the added entropy term in \eqref{V^pi}. To overcome the technical hurdles, we propose a plan (the ``grand plan'' in Section~\ref{subsec:grand plan}) to find a uniform estimate of $\{v^n\}_{n\in\N}$. The plan's implementation requires new Sobolev estimates designed specifically for the purpose of policy iteration (Section~\ref{subsec:Sobolev}) and a nontrivial containment of the entropy growth (Section~\ref{subsec:entropy}). Based on this, Proposition~\ref{thm.the.uniformbound} gives a uniform bound for $\{v^n\}_{n\in\N}$ and their derivatives, which allows us to prove $v^*=V^*$ and obtain an optimal relaxed control; see Theorem \ref{thm.verification.limit}.

First, under mild regularity conditions, we find that for any relaxed control $\pi$ given by \eqref{pi=Gamma}, 
\be\label{pit}
\tilde{\pi}(x,u):=\Gamma(x,D_x V^\pi(x),u),\quad \forall (x,u)\in\R^d\times U,
\ee
is again a well-defined relaxed control and it outperforms $\pi$, i.e., $V^{\pit}\ge V^\pi$. 

\begin{Proposition}\label{prop.policy.improvement}
	Let $\Lambda_1<\infty$ in \eqref{eq.parameters.m} and assume \eqref{eq.assume.ellip}. For any $p:\R^d\rightarrow \R^d$ of the class $\Cc^{0,1}_{\text{unif}}(\R^d)$, consider $\pi$ given by \eqref{pi=Gamma} and recall that $\pi\in \Ac_M$ (Proposition~\ref{prop.sde}) and $V^\pi\in \Cc^{2,\alpha}_{\text{unif}}(\R^d)$ for all $0<\alpha<1$ (Proposition~\ref{prop.Vpi}). Then, $\pit$ given by \eqref{pit} is well-defined and satisfies $\tilde\pi\in\Ac_M$, $V^{\tilde\pi}\in\Cc^{2,\alpha}_{\text{unif}}(\R^d)$ for all $0<\alpha<1$, and $V^{\tilde{\pi}}(x)\geq V^\pi(x)$ for all $x\in \R^d$.
\end{Proposition}

\begin{proof}
	With $D_x V^\pi\in \Cc^{1,\alpha}_{\text{unif}}(\R^d)\in \Cc^{0,1}_{\text{unif}}(\R^d)$, Proposition~\ref{prop.sde} (with $p=D_x V^\pi$) asserts $\tilde{\pi}\in \Ac_M$. Applying Proposition \ref{prop.Vpi} (with $k=0$ and $p=D_x V^\pi$) then yields $V^{\tilde{\pi}}\in \Cc^{2,\alpha}_{\text{unif}}(\R^d)$ for all $0<\alpha<1$. For any $x\in\R^d$, by applying It\^{o}'s formula to $e^{-\rho t} V^\pi(X^{\tilde{\pi}}_t)$ and recalling \eqref{eq.elliptic.gibbs}, we get
	\bee
	\begin{aligned}
		V^\pi(x) &= e^{-\rho T}V^{\pi}(X^{\pit}_T)-\int_0^T e^{-\rho t}\Lc^{\pit}V^\pi(X^{\pit}_t)dt-\int_0^T e^{-\rho t} D_xV^{\pi}(X^{\pit}_t)\cdot \sigma(X^{\pit}_t) dW_t,\quad \forall T>0.
	\end{aligned}
	\eee
	As $b,\sigma, V^\pi, D_xV^{\pi}$, and $D^2_{x}V^{\pi}$ are bounded (by $\Lambda_1<\infty$ and $V^\pi\in \Cc^2(\R^d)$), when we take expectation on both sides of the above equality and let $T\to\infty$, the dominated convergence theorem gives $V^\pi(x)= -\E\left[\int_0^\infty e^{-\rho t} \Lc^{\pit}V^\pi(X^{\pit}_t) dt\right]$. This, along with $V^{\pit}(x) = \E[\Delta(X^{\pit}_\cdot)]$ as in \eqref{V^pi=E}, yields
	\be \label{simi =} 
	\begin{aligned}
		V^{\tilde{\pi}}(x)-V^\pi(x) 
		&=\E\bigg[\int_0^\infty e^{-\rho t}\Big( \Lc^{\pit}V^\pi(X^{\pit}_t)+\hat{r}(X^{\pit}_t, D_x V^\pi(X^{\pit}_t)) 
		-\hat{\Hc}(X^{\pit}_t, D_x V^\pi(X^{\pit}_t)) \Big) dt\bigg] \\
		&=\E\bigg[\int_0^\infty e^{-\rho t}\bigg\{-\rho V^\pi(X^{\pit}_t)+\frac{1}{2}\tr(\sigma\sigma' D^2_x V^\pi(X^{\pit}_t))\\
		&\quad + \int_U\Big( b(u, X^{\pit}_t)\cdot D_x V^\pi(X^{\pit}_t)+r(X^{\pit}_t, u)-\lambda \ln\pit(X^{\pit}_t,u)  \Big) \pit(X^{\pit}_t,u) du\bigg\} dt\bigg]
	\end{aligned}
\ee
	where the second equality above follows from \eqref{eq.elliptic.gibbs}, \eqref{eq.hat.notation}, and \eqref{eq.h}. By the definition of $\tilde{\pi}(x,u) =\Gamma(x,D_x V^\pi(x),u)$ in \eqref{eq.gibbs.pi}, the integral in the last line of \eqref{simi =} is larger or equal to the same integral with $\tilde\pi(X^{\pit}_t,u)$ therein replaced by $\pi(X^{\pit}_t,u)$. Hence, with $\pi(X^{\pit}_t,u)$ in place of $\pit(X^{\pit}_t,u)$ in the last line of \eqref{simi =}, a calculation similar to that in \eqref{simi =} gives
	\bee
	\begin{aligned}
		V^{\tilde{\pi}}&(x)-V^\pi(x)
		\ge \E\bigg[\int_0^\infty e^{-\rho t}\Big( \Lc^\pi V^\pi(X^{\pit}_t)+\hat{r}(X^{\pit}_t, p(X^{\pit}_t))-\hat{\Hc}(X^{\pit}_t, p(X^{\pit}_t))\Big) dt\bigg]=0,
	\end{aligned}
	\eee
	where the last equality holds because $V^\pi$ satisfies \eqref{eq.HJB.new.z} (by Proposition~\ref{prop.Vpi}). 
\end{proof}

\begin{Corollary}\label{coro.PIA.works}
	Let $\Lambda_1<\infty$ in \eqref{eq.parameters.m} and assume \eqref{eq.assume.ellip}. Then, in the PIA \eqref{PIA}, we have $\pi^n\in\Ac_M$ and $v^{n+1}\ge v^{n}$ for all $n\in\N$. 
	Hence, 
	\begin{equation}\label{v^*}
		v^*(x):= \lim_{n\rightarrow \infty} v^n(x)\quad \forall x\in\R^d
	\end{equation}
	is well-defined (i.e., exists and is finitely-valued). 
\end{Corollary}

\begin{proof}
	For each $n\in\N$, by Corollary~\ref{prop.vn.bound} (with $\bar k=1$), we have $\pi^n\in\Ac_M$ and $v^n\in \Cc^{2,\alpha}_{\text{unif}}(\R^d)$ for all $0<\alpha<1$. It then follows from Proposition~\ref{prop.policy.improvement} that $v^{n+1} \ge v^n$. As the sequence $\{v^n\}_{n\in\N}$ is monotone and has a uniform upper bound (by Lemma~\ref{lem:V^pi bdd}), $v^*$ in \eqref{v^*} is well-defined. 
\end{proof}

Corollary~\ref{coro.PIA.works} shows that, as long as $\Lambda_1<\infty$ and \eqref{eq.assume.ellip} holds, the PIA \eqref{PIA} functions properly: it gives rise to a better relaxed control (which generates a larger value) in each iteration, and ultimately points to the largest possible value $v^*$ that cannot be further improved by the algorithm.  
It is then natural 
to ask whether $v^*$ coincides with the optimal value $V^*$ in \eqref{V^*}. 


\subsection{The Grand Plan}\label{subsec:grand plan}
To show that the sequence $\{v^n\}_{n\in\N}$ in the PIA \eqref{PIA} will ultimately lead to the optimal value (i.e., $v^*=V^*$), we will focus on proving 
\be\label{eq.thm.uniform} 
\sup_{n\in\N} \| v^n\|_{\Cc^2(\R^d)} <\infty.
\ee	
Once this is established (Proposition~\ref{thm.the.uniformbound}), the Arzela-Ascoli theorem and a modified verification argument will ensure $v^*=V^*$ and provide an optimal relaxed control (Theorem~\ref{thm.verification.limit}).   

While $\| v^n\|_{\Cc^0(\R^d)}$ is known to be uniformly bounded (by Lemma~\ref{lem:V^pi bdd} and Corollary~\ref{coro.PIA.works}), it is difficult to estimate $\|D_x v^n\|_{\Cc^1(\R^d)} = \|D_x v^n\|_{\Cc^0(\R^d)} + \|D_x^2 v^n\|_{\Cc^0(\R^d)}$, due to the added entropy term in \eqref{V^pi}. To circumvent the involved technical challenges, we propose the following plan. 

Take $s := \lfloor d/2\rfloor+1$. For $n\in\N$ large enough, we perform the following estimations in order:
\begin{itemize}
	\item [1.]  For each $x\in\R^d$, bound $|D_x v^n(x)| + |D_x^2 v^n(x)|$ by $\|v^n\|_{W^{s+2,2}(B_1(x))}$ (Lemma~\ref{lm.infinity.emded}). 
	\item [2.] Bound $\|v^n\|_{W^{s+2,2}(B_1(x))}$ by a power function of the sum of 
	\begin{equation}\label{L^q terms}
		\|v^{n-i}\|_{L^{q}(B_\eta(x))},\ i=0,1,\ldots,s-1,\quad \hbox{and}\quad \|\hat{\Hc}(\cdot,D_x v^{n-s-1}(\cdot))\|_{L^q(B_\eta(x))}
	\end{equation}
	for some $q\ge 1$ and $\eta>0$ (see \eqref{eq.vn.h0norm0} and \eqref{eq.vn.h0norm0'''}, which reply on Lemmas \ref{lm.hnorm.estimate} and  \ref{lm.fn.vn-1}). 
	\item [3.] Bound $|\hat{\Hc}(\cdot,D_x v^{n-s-1}(\cdot))|$ by $\psi(|D_x v^{n-s-1}(\cdot)|)$, for a suitable function $\psi$ of {\it logarithmic} growth (Lemma~\ref{lm.1} and Corollary~\ref{coro:H bdd}). 
	\item [4.] Combine the local estimates in Steps 1 and 2. By Step 3 and $\sup_{N\in\N} \| v^n\|_{\Cc^0(\R^d)}<\infty$, turn the local estimate of $|D_x v^n(x)| + |D_x^2 v^n(x)|$ into a global one (i.e., independent of $x\in\R^d$), thereby obtaining a uniform bound for $\|D_x v^n\|_{\Cc^1(\R^d)}$. 
\end{itemize} 

To properly appreciate the above procedure, we point out two important observations.

\begin{Remark}\label{rem:sublinear important}
	The added entropy term in \eqref{V^pi} (i.e., $\hat{\Hc}(\cdot,D_x v^{n-s-1}(\cdot))$ therein) shows up in Step 2. It is critical that it has logarithmic growth in $|D_x v^{n-s-1}(\cdot)|$, as stated in Step 3. Since any power of a logarithmic function remains sublinear, by plugging the estimate in Step 3 into the power function of Step 2, we can see $\|v^n\|_{\Cc^1(\R^d)}$ dominated by $\|v^{n-s-1}\|_{\Cc^1(\R^d)}$. This readily suggests \eqref{eq.thm.uniform}, and is made rigorous in the proof of Proposition~\ref{thm.the.uniformbound}; see \eqref{eq.thm.key} and the arguments below it. 
\end{Remark}

\begin{Remark}\label{rem:detour}
	At first glance, Steps 1--4 above seem a puzzling ``detour'': to obtain a uniform bound for $\|D_x v^n\|_{\Cc^1(\R^d)} = \|D_x v^n\|_{\Cc^0(\R^d)}+\|D^2_x v^n\|_{\Cc^0(\R^d)}$,  
	we look for a local $W^{s+2,2}$ bound, transform it into a local $L^q$ bound, and finally switch back to the original $\Cc^1(\R^d)$ space. It is natural to wonder if one can simply work within $\Cc^1(\R^d)$ and derive the desired uniform bound 
	by the standard Schauder estimates for elliptic equations. As we will explain subsequently, such a detour may be necessary, precisely because of the presence of entropy regularization; see Remark \ref{rem:detour necessary} 
	for details. 
\end{Remark}

To begin with Steps 1 and 2 of our plan, we will derive some useful estimates in Sobolev spaces. 


\subsection{Sobolev Estimates}\label{subsec:Sobolev}
A standard result in Sobolev spaces (see e.g., \cite[Exercise 18, page 292]{Evans-book-98}) states that for any $w\in W^{s,2}(\R^d)$ with $s>d/2$, $w$ is bounded and satisfies  
\begin{equation}\label{Evans}
	\|w\|_{L^\infty(\R^d)} \le C \|w\|_{W^{s,2}(\R^d)}, 
\end{equation}
where $C>0$ is a constant depending on only $s$ and $d$. For the case where the derivatives of $w$ do not necessarily lie in $L^2(\R^d)$, we can generalize \eqref{Evans} to the following localized version.

\begin{Lemma}\label{lm.infinity.emded}
	For any $w\in \Cc^s(\R^d)$ with $s>d/2$, there exists $C>0$ depending on only $s$ and $d$ such that 
	$
	|w(x)|\leq C \|w\|_{W^{s,2}(B_1(x))}$ for all $x\in \R^d$.
\end{Lemma}

\begin{proof}
For any $x\in \R^d$, consider the cut-off function $\xi\in \Cc^\infty(\R^d)$ defined by $\xi(z):= \exp\big(\frac{|z-x|^2}{|z-x|^2-1}\big)$ for $z\in B_1(x)$ and $\xi\equiv 0$ on $\R^d\setminus B_1(x)$. Note that $\xi(x)=1$ and $0<\xi \leq 1$ on $B_1(x)$. It follows that
\begin{equation}\label{xi u}
	\| \xi w\|_{W^{s,2}(\R^d)}=\| \xi w\|_{W^{s,2}(B_1(x))}\leq K  \|w\|_{W^{s,2}(B_1(x))},
\end{equation}
where $K>0$ is a constant depending on the $L^\infty(B_1(x))$-norm of the derivatives of $\xi$ up to the $s^{th}$ order. By the definition of $\xi$, this $L^\infty(B_1(x))$-norm is independent of $x\in\R^d$. Hence, $K>0$ is independent of $x\in\R^d$ and depends on only $d$ and $s$. Finally, observe that 	
$
|w(x)|\leq  \|\xi w\|_{L^\infty(\R^d)}\leq C\|\xi w\|_{W^{s,2}(\R^d)}\leq CK  \|w\|_{W^{s,2}(B_1(x))},
$
where the second inequality stems from \eqref{Evans} (applicable as \eqref{xi u} implies $\xi w\in W^{s,2}(\R^d)$) and the third inequality is due to \eqref{xi u}. The desired result follows as $C$ and $K$ depend on only $d$ and $s$. 
\end{proof}

Lemma~\ref{lm.infinity.emded} provides a local Sobolev estimate of $w$ with a potentially large order $s>d/2$. Our goal now is to reduce the order to a sufficiently smaller one.
We first note that the reduction from order $2$ to 0 (i.e., from $W^{2,q}$ to $W^{0,q}=L^q$) is readily known, when $w$ is a solution to an elliptic equation. 
Specifically, given a bounded connected open domain $E\subset \R^d$, recall the elliptic operator $L^{(\gamma,\beta,\delta)}$ in \eqref{L}. 
As long as $\beta\in\Cc^0(E)$, $\delta\in\Cc^0(\overline E)$, and \eqref{tr bdd} is satisfied, for any fixed $q\ge 1$, a classical interior $W^{2,q}$ estimate holds for any $w\in W^{2,q}(E)$ that satisfies \eqref{L=f} a.e.: 
for any $E'\subset\subset E$, 
\begin{equation}\label{W^2,q}
	\|w\|_{W^{2,q}(E')} \le C (\|w\|_{L^q(E)} + \|f\|_{L^q(E)}),
\end{equation}
where $C>0$ is a constant that depends on only $d$, $q$, $K$ in \eqref{tr bdd}, $\text{dist}(E',\partial E)$, $\gamma$, $\|\beta\|_{\Cc^{0}(E)}+\|\delta\|_{\Cc^{0}(E)}$, and the modulus of continuity of $\delta_{ij}$; see e.g., \cite[Theorem 4.2, page 47]{chen1998second}.

In the following, we will derive a higher-order counterpart of \eqref{W^2,q}: for any $k\in \N_0$, $W^{2,q}$ and $L^q$ in \eqref{W^2,q} will be replaced by $W^{k+2, q}$ and $W^{k, q'}$ (or $L^{q'}$), respectively, for some $q'>q$. In so doing, we find that the constant $C>0$ will depend on not only $\|\beta\|_{\Cc^{0}(E)}$,  but also $\|\beta\|_{W^{k,q'}(E)}$. This additional dependence on $\|\beta\|_{W^{k,q'}(E)}$ is spelled out explicitly in \eqref{eq.lemma.wkest} below, our ultimate higher-order interior estimate that reduces order from $k+2$ to $k$. On one hand, this sets us apart from standard interior estimates in the PDE literature, where dependence on the coefficients of a PDE (including $\beta$) is incorporated into a constant multiplier $C>0$; on the other hand, such spelled-out explicit dependence is necessary for Proposition~\ref{thm.the.uniformbound} below; see Remark~\ref{rem:explicit beta} (i) for details. 

\begin{Lemma}\label{lm.hnorm.estimate}
	Consider the elliptic operator $L^{(\gamma,\beta,\delta)}$ in \eqref{L} with $E=\R^d$ and suppose that \eqref{tr bdd} holds.
	Then, there exist three nondecreasing functions 
	\begin{align*}
		&A(k,q): \N_0\times[1,\infty)\rightarrow \R_+\ \hbox{satisfying}\  A(k,q)\ge q,\\
		&Q(k,q):  \N_0 \times[1,\infty)\rightarrow \N_0,\\
		&R(k,\eta): \N_0\times \R_+\rightarrow \R_+\ \hbox{satisfying}\  R(k,\eta)\geq 2 \eta,
	\end{align*}
	such that whenever $\beta\in \Cc^{k}(\R^d)$ and $\delta \in \Cc^{k\vee 1}(\R^d)$ for some $k\in\N_0$, the following holds: given $f\in \Cc^k(\R^d)$, any solution $w\in \Cc^{k+2}(\R^d)$ to \eqref{L=f} (with $E=\R^d$) 
	satisfies
	\be\label{eq.lemma.wkest} 
	\begin{aligned}
		\|w \|_{W^{k+2,q}(B_\eta(x))}\leq & H_k\left(1+\|\beta\|_{W^{k,A(k,q)}(B_{R(k,\eta)}(x))}\right)^{Q(k,q)}\\
		&\cdot \left(\|w\|_{L^{A(k,q)}(B_{R(k,\eta)}(x))} +\|f\|_{W^{k,A(k,q)}(B_{R(k,\eta)}(x))}\right),
	\end{aligned}
\ee 
	for all $x\in\R^d$ and $(q,\eta)\in[1,\infty)\times \R_+$, where $H_k>0$ is a constant that depends on only $k$, $K$ in \eqref{tr bdd}, $q$, $\eta$, $d$, $\gamma$, and $\Phi_k := \|\beta\|_{\Cc^{0}(\R^d)}+\|\delta\|_{\Cc^{k\vee 1}(\R^d)}$.
\end{Lemma}

\begin{proof}
	We prove the desired result by induction on $k\in \N_0$. For the case $k=0$, by taking 
	\begin{equation}\label{k=0 case}
		A(0,q) = q,\ \ Q(0,q)=0,\ \ \hbox{and}\ \ R(0,\eta)=2\eta\quad \forall (q,\eta)\in [1,\infty)\times \R_+, 
	\end{equation}
	we observe from the classical estimate \eqref{W^2,q} that \eqref{eq.lemma.wkest} holds, with $H_0>0$ a constant depending on only $K$, $q$, $\eta$, $d$, $\gamma$, and $\Phi_0$. 
	Now, suppose that for some $s\in\N$ such that $\beta,\delta\in \Cc^{s}(\R^d)$, \eqref{eq.lemma.wkest} holds for all $k\leq s-1$, with $A(k,q)$, $Q(k,q)$, $R(k,\eta)$, and $H_k$ all known for $k\leq s-1$ and $(q, \eta)\in [1,\infty)\times \R_+$. Our goal is to establish \eqref{eq.lemma.wkest} for $k=s$. 
	
	For any $f\in \Cc^s(\R^d)$, let $w\in \Cc^{s+2}(\R^d)$ be a solution to \eqref{L=f} (with $E=\R^d$). Fix $x\in\R^d$ and $(q, \eta)\in [1,\infty)\times \R_+$ and let $R(s,\eta)\ge 2 \eta$ be a number whose precise value will be determined subsequently. For $j=1,\ldots,d$, a direct calculation shows that $D_{x_j} w\in\Cc^{s+1}(\R^d)$ satisfies
	\be\label{eq.lu.extend}  
	\begin{aligned}
		L^{(\gamma,\beta,\delta)} (D_{x_j} w) &=D_{x_j}(L^{(\gamma,\beta,\delta)} w)-\tr(D_{x_j}(\delta\delta') D^2_x w)-D_{x_j} \beta \cdot D_x w\\
		&=D_{x_j}f-\tr(D_{x_j}(\delta\delta') D^2_x w)-D_{x_j}\beta \cdot D_x w =: \tilde{f}\in  \Cc^{s-1}(\R^d).
	\end{aligned}
	\ee
	Hence, by the induction hypothesis (particularly, \eqref{eq.lemma.wkest} with $k=s-1$), we have 
	\be\label{eq.lu.normstart}   
	\begin{aligned}
	\| D_{x_j} w\|_{W^{s+1,q}(B_\eta(x))}
&	\leq  H_{s-1}\left(1+\|\beta\|_{W^{{s-1},A(s-1,q)}(B_{R(s-1,\eta)}(x))}\right)^{Q(s-1,q)}\\
		&\hspace{0.15in}  \cdot\left(\|D_{x_j}w\|_{L^{A(s-1,q)}(B_{R(s-1,\eta)}(x))} +\|\tilde{f}\|_{W^{s-1,A(s-1,q)}(B_{R(s-1,\eta)}(x))}\right).
	\end{aligned}
\ee 
	In the following, we focus on estimating the right-hand side of \eqref{eq.lu.normstart}. For simplicity, we set
	\begin{equation}\label{bqbeta}
		\bar q := A(s-1,q)\ge q\ge 1\quad \hbox{and}\quad \bar \eta := R(s-1,\eta)\ge 2\eta>0.
	\end{equation}
	
	First, by the induction hypothesis (particularly, \eqref{eq.lemma.wkest} with $(k, q,\eta) = (0,\bar q,\bar\eta)$) and \eqref{k=0 case}, 
	\be\label{eq.dw.l} 
	\begin{aligned}
		\|D_{x_j} w\|_{L^{\bar q}(B_{\bar \eta}(x))}\leq 	\|w\|_{W^{2,\bar q}(B_{\bar \eta}(x))}
		\le \widetilde H_0 \Big(\| w\|_{L^{\bar q}(B_{2\bar\eta}(x))}+\|f\|_{L^{\bar q}(B_{2\bar\eta}(x))} \Big),
	\end{aligned}
	\ee 
	where $\widetilde H_0>0$ is a constant depending on only 
	$K$, $d$, $\gamma$, $\Phi_{0}$, and $(\bar q, \bar\eta)$, which in turn depends on $s$, $q$, and $\eta$. On the other hand, by the definition of $\tilde{f}$ in \eqref{eq.lu.extend}, 
	\be\label{tf-1} 
	\begin{aligned}
		\|\tilde{f}\|_{W^{s-1,\bar q}(B_{\bar\eta}(x))} \leq &\|f\|_{W^{s,\bar q}(B_{\bar\eta}(x))} + C_\delta \|w\|_{W^{s+1,\bar q}(B_{\bar\eta}(x))}
		+  \|D_{x_j} \beta \cdot D_x w\|_{W^{s-1,\bar q}(B_{\bar\eta}(x))},
	\end{aligned}
\ee 
	where $C_\delta>0$ is a constant depending on only $\|\delta\|_{\Cc^s(\R^d)}$. Moreover, by H\"{o}lder's inequality, the last term on the right-hand side of \eqref{tf-1} satisfies
	\be\label{eq.norm.prod}   
	\begin{aligned}
		&	\|D_{x_j} \beta \cdot D_x w\|_{W^{s-1,\bar q}(B_{\bar\eta}(x))} \leq  C \|\beta\|_{W^{s,2\bar q}(B_{\bar\eta}(x))} \| w\|_{W^{s,2\bar q}(B_{\bar\eta}(x))},
	\end{aligned}
	\ee
	where $C>0$ is a constant depending on only $s$ and $d$. By plugging \eqref{eq.dw.l}, \eqref{tf-1}, and \eqref{eq.norm.prod} into \eqref{eq.lu.normstart} and using the fact  
	\begin{equation}\label{a fact}
		\|\cdot\|_{W^{a,c}(E)}\le \|\cdot\|_{W^{\hat a,\hat c}(E)},\quad \hbox{if $a\le \hat a$, $c\le \hat c$, and $E\subset \R^d$ is bounded}, 
	\end{equation}
	we obtain
	\be\label{eq.lu.norm}  
	\begin{aligned}
	\| D_{x_j} w\|_{W^{s+1,q}(B_\eta(x))}
	&	\leq   H^* \left(1+\|\beta\|_{W^{{s-1},\bar q}(B_{\bar\eta}(x))}\right)^{Q(s-1,q)}
	 \cdot\Big( \|w\|_{L^{\bar q}(B_{2\bar\eta}(x))} 
		+\|f\|_{W^{s,\bar q}(B_{2\bar\eta}(x))} \\
		& \hspace{0.17in} + \left(1+\|\beta\|_{W^{s,2\bar q}(B_{\bar\eta}(x))}\right) \|w\|_{W^{s+1,2\bar q}(B_{\bar\eta}(x))}\Big),
	\end{aligned}
	\ee
	where $H^* :=2 H_{s-1}(1+\widetilde H_0 +C_\delta+C) >0$ depends on only $s$, $K$, $q$, $\eta$, $d$, $\gamma$, and $\Phi_s$. Let us further estimate the last term in the above inequality, by noting that the induction hypothesis (particularly, \eqref{eq.lemma.wkest} with $(k, q,\eta) = (s-1, 2 \bar q, \bar\eta)$) implies
	\be\label{eq.lowerorder.norm}  
	\begin{aligned}
		\|w\|_{W^{s+1,2\bar q}(B_{\bar\eta}(x))} &\leq  \widetilde H_{s-1} \Big( 1+\|\beta\|_{W^{s-1,A(s-1, 2\bar q)}(B_{R(s-1,\bar\eta)}(x))}\Big)^{Q(s-1,2\bar q)}\\
		& \quad \cdot \Big(\|w\|_{L^{A(s-1,2\bar q)}(B_{R(s-1,\bar\eta)}(x))} +\|f\|_{W^{s-1,A(s-1,2\bar q)}(B_{R(s-1,\bar\eta)}(x))} \Big),
	\end{aligned}
	\ee
	where $\widetilde H_{s-1}>0$ is a constant depending on only $k=s-1$, $K$, $d$, $\gamma$, $\Phi_{s-1}$, and $(2\bar q, \bar\eta)$, which in turn depends on $s$, $q$, and $\eta$.
	Now, let us define 
	\begin{equation}\label{eq.a.R} 
		\begin{split}
			A(s,q) &:= A(s-1,2\bar q)\ge 2 \bar q,\quad R(s,\eta) := R(s-1,\bar\eta)\ge 2\bar\eta,\\
			Q(s,q) &:= Q(s-1,q)+Q(s-1, 2\bar{q})+1, 
		\end{split}
	\end{equation}
	and consider the constant $H_s := 2H^*(1+\widetilde H_{s-1})>0$, which depends on only $s$, $K$, $q$, $\eta$, $d$, $\gamma$, and $\Phi_s$. Then, by plugging \eqref{eq.lowerorder.norm} into \eqref{eq.lu.norm} and using \eqref{a fact}, we get
	\begin{align*} 
		&\| D_{x_j} w\|_{W^{s+1,q}(B_\eta(x))}\\
		&\leq  H^* \Big(1+\|\beta\|_{W^{{s-1},\bar q}(B_{\bar\eta}(x))}\Big)^{Q(s-1,q)}\\
		&\hspace{0.3in} \cdot \Big( \|w\|_{L^{\bar q}(B_{2\bar\eta}(x))} 
		+\|f\|_{W^{s,\bar q}(B_{2\bar\eta}(x))} + \widetilde H_{s-1}\left(1+\|\beta\|_{W^{s,A(s,q)}(B_{R(s,\eta)}(x))}\right)^{Q(s-1,2\bar q)+1}\\
		&\hspace{0.55in} \cdot \left(\|w\|_{L^{A(s,q)}(B_{R(s,\eta)}(x))} + \|f\|_{W^{s,A(s, q)}(B_{R(s,\eta)}(x))}\right)\Big)\\
		&\leq H_s \Big(1+\|\beta\|_{W^{s,A(s,q)}(B_{R(s,\eta)}(x))}\Big)^{Q(s,q)} \Big(\|w\|_{L^{A(s,q)}(B_{R(s,\eta)}(x))} + \|f\|_{W^{s,A(s, q)}(B_{R(s,\eta)}(x))}\Big).
	\end{align*}
	Summing this up over all $j=1,2,\ldots,d$ gives \eqref{eq.lemma.wkest} for $k=s$.
\end{proof}

\begin{Remark}\label{rem:explicit beta}
	In Proposition~\ref{thm.the.uniformbound} below, we will apply \eqref{eq.lemma.wkest} recursively to $\{v^n\}_{n\in\N}$ in the PIA \eqref{PIA}. Specifically, $\beta$ and $f$ in \eqref{eq.lemma.wkest} will respectively play the roles of $\hat{b}(\cdot,D_x v^{n-1}(\cdot))$, the controlled drift, and $\hat{r}(\cdot , D_x v^{n-1}(\cdot))-\hat{\Hc}(\cdot,D_x v^{n-1}(\cdot))$, the reward plus entropy; see \eqref{fmbm}-\eqref{eq.vn.fn} below. Note that
	\begin{enumerate}[\normalfont(i)]
		\item As $\beta =\hat{b}(\cdot,D_x v^{n-1}(\cdot))$ changes across different iterations, it is important that \eqref{eq.lemma.wkest} spells out the explicit dependence on $\beta$, so that the estimate of $v^n$ can be made precise, in terms of the lower-order Sobolev norms of $\{v^m\}_{m\le n}$; see \eqref{eq.vn.h0norm0} below. 
		\item
		The explicit formulas of functions $A(k,q)$, $Q(k,q)$, $R(k,\eta)$ in \eqref{eq.lemma.wkest} 
		are not needed, although \eqref{bqbeta} and \eqref{eq.a.R} above 
		already provide a recursive formula to precisely define these functions.
	\end{enumerate}
\end{Remark}

In view of Remark~\ref{rem:explicit beta}, for the eventual application of \eqref{eq.lemma.wkest} in Proposition~\ref{thm.the.uniformbound}, we still need to compute the $W^{k,q}$-norms of $\hat{b}(\cdot,p(\cdot))$ and $\hat{r}(\cdot , p(\cdot))-\hat{\Hc}(\cdot,p(\cdot))$. 

\begin{Lemma}\label{lm.fn.vn-1}
	Fix $q\in[1,\infty)$. Suppose that there exists $k\in \N_0$ such that $\Lambda_k<\infty$ in \eqref{eq.parameters.m} and $p:\R^d\rightarrow \R^d$ is of the class $\Cc^{k}(\R^d)$. Then, for any $x_0\in\R^d$ and $\eta>0$, the function $f(x,y,u) := r(x,u)-\Hc(x, y,u)$ (or, $f(x,u) := b(x,u)$) satisfies
	$$ 
	\|\hat{f}(\cdot,p(\cdot))\|_{W^{k,q}(B_\eta(x_0))}\leq C_{k} \left(1+\|p\|_{W^{k,(k+1)2^{k+1}q }(B_\eta(x_0))}\right)^{(k+1)2^{k+1}},
	$$
	where $C_k>0$ is a constant that depends on only $k$, $\Lambda_k$, $\lambda$, $d$, $q$, and $\eta$.
\end{Lemma}

\begin{proof}
In the following, $C_k>0$ is a generic constant that may change from line to line but depends on only $k$, $\Lambda_k$, $\lambda$, $d$, $q$, and $\eta$. By \eqref{eq.lips.xy}, 
for any multi-index $a=(a_1,\ldots,a_d)$ with $|a|_{l_1}\in \{0,1,\ldots,k\}$,
\begin{align*}
	&\| D^a_x \hat{f}(\cdot,p(\cdot))\|_{L^q(B_\eta(x_0))} \\
	&\leq  C_k \bigg\{\int_{B_\eta(x_0)} \left(1+|p(x)|^{(k+1)2^k}\right)^q \bigg(1+\sum_{|c|_{l_1}=0,\ldots,k}|D^c_x p(x)|\bigg)^{kq} dx\bigg\}^{1/q}\\
	&\le C_k \bigg\{\int_{B_\eta(x_0)} \Big(1+|p(x)|\Big)^{(k+1)2^k q} \bigg(1+\sum_{|c|_{l_1}=0,\ldots,k}|D^c_x p(x)|\bigg)^{(k+1)2^k q} dx\bigg\}^{1/q}\\
	& \leq  C_k\bigg\{\int_{B_\eta(x_0)} \left(1+ |p(x)|^{(k+1)2^{k+1}q}\right) dx \bigg\}^{\frac{1}{2q}}
 \cdot \bigg\{\int_{B_\eta(x_0)} \bigg(1+\sum_{|c|_{l_1}=0,\ldots,k}|D^c_x p(x)|^{(k+1)2^{k+1}q}\bigg) dx\bigg\}^{\frac{1}{2q}}\\
	&\le C_k \bigg\{ \left(\Leb(B_\eta(x_0))+ \|p\|^{(k+1)2^{k+1}q}_{W^{k,(k+1)2^{k+1}q}(B_\eta(x_0))}\right) \bigg\}^{\frac{1}{q}}\leq  C_k\left(1+ \|p\|_{W^{k,(k+1)2^{k+1}q}(B_\eta(x_0))}\right)^{(k+1)2^{k+1}},
\end{align*}
where the third inequality follows from H\"{o}lder's inequality and $(b_1+b_2)^j\leq 2^{j-1}({b_1}^j+b_2^j)$ for $b_1,b_2\in\R$ and $j\ge 1$. Note that $C_k$ in the first inequality above (obtained from \eqref{eq.lips.xy} in Lemma~\ref{lemma.lips.xy} (i)) depends on only $k$, $\Lambda_k$, $\lambda$, and $d$; it gets additional dependence on $q$ in the third line, and on $\eta$ in the last line. Summing up the above inequality over all $a$ with $|a|_{l_1}\in\{0,1,\ldots,k\}$ yields the desired result.
\end{proof}


\subsection{Logarithmic Growth of the Entropy}\label{subsec:entropy}
Now, we proceed to Step 3 of our plan specified in Section~\ref{subsec:grand plan}. While the entropy term in \eqref{V^pi} has a natural upper bound (i.e.,  $\ln(\Leb(U))$) thanks to \eqref{eq.kl.lbound}, it is generally unclear how to control its behavior on the lower end. To obtain an appropriate lower bound that grows only logarithmically, we will impose additional regularity on the functions $r$ and $b$ as well as the action space $U$. 

\begin{Assumption}\label{assume.u.lips}
	The maps $u\mapsto r(x,u)$ and $u\mapsto b(x,u)$ are Lipschitz, uniformly in $x$; that is, 
	\begin{equation}\label{Theta}
		\Theta:= \sup_{u_1,u_2\in U, x\in \R^d}\frac{|r(x,u_1)-r(x,u_2)|+|b(x,u_1)-b(x,u_2)|}{|u_1-u_2|}<\infty.
	\end{equation}
\end{Assumption}

We will also assume that the action space $U\subset \R^\ell$ fulfills a {\it uniform cone condition}. To properly state the assumption, note that when $\ell>1$, for any $\alpha\in [0,\pi/2]$, 
$$
\Delta_\alpha := \{u=(u_1,\ldots,u_\ell)\in\R^\ell : u_1^2+\ldots+u_{\ell-1}^2\leq \tan^2(\alpha) u_{\ell}^2\; \text{and}\; u_{\ell}\geq 0\}
$$
is a cone with vertex $0\in\R^\ell$, axis $u_1=u_2=\ldots=u_{\ell-1}=0$, and angle $\alpha$. Given $u\in\R^\ell$, the region obtained by a rotation of $u+\Delta_\alpha$ in $\R^\ell$ about $u$ is called {\it a cone with vertex $u$ and angle $\alpha$}. 

\begin{Assumption}\label{assume.U.cone}
When $\ell>1$, there exist $\zeta>0$ and $\alpha\in (0, \pi/2]$ such that for any $u\in U$, there is a cone with vertex $u$ and angle $\alpha$ (denoted by $\operatorname{cone}(u,\alpha)$) that satisfies $\left(\operatorname{cone}(u,\alpha)\cap B_{\zeta}(u)\right)\subseteq U$.
When $\ell=1$, there exists $\zeta>0$ such that for any $u\in U$, either $[u-\zeta,u]$ or $[u, u+\zeta]$ is contained in $U$.
\end{Assumption}

\begin{Remark}
Assumption~\ref{assume.U.cone} states that a cone with a fixed size (determined by the fixed length $\zeta>0$ and angle $\alpha\in (0, \pi/2]$) can be fitted inside $U$ at any $u\in U$. 
Any convex domain of $\R^\ell$ with a positive Lebesgue measure fulfills Assumption~\ref{assume.U.cone}, and so does a finite disjoint union of such domains.
\end{Remark}

\begin{Remark}
The standing condition ``$0<\Leb(U)<\infty$'' and Assumption~\ref{assume.U.cone} imply that $U$ is bounded. Indeed, if $U$ is not bounded, there exists $u_n\in U\setminus B_n(0)$ for all $n\in \N$. By Assumption~\ref{assume.U.cone}, we have $\big(\operatorname{cone}(u_n,\alpha)\cap B_{\zeta\wedge 1/2}(u_n)\big)\subseteq U$, where $\operatorname{cone}(u_n,\alpha)$ is a cone with vertex $u_n$ and angle $\alpha$. As $\{\operatorname{cone}(u_n,\alpha)\cap B_{\zeta\wedge 1/2}(u_n)\}_{n\in \N}$ is an infinite collection of disjoint sets of a fixed size, we must have $\Leb(U)=\infty$, a contradiction.
\end{Remark}

Under the above two assumptions, we find that the Gibbs-form function $\Gamma(x,y,u)$ in \eqref{eq.gibbs.pi} has polynomial growth in $y$, uniformly in $x$ and $u$.

\begin{Lemma}\label{lm.1}
	Suppose that Assumptions \ref{assume.u.lips} and \ref{assume.U.cone} hold. Then, the function $\Gamma$ in \eqref{eq.gibbs.pi} satisfies 
	\[
	\Gamma(x,y,u) \le C^\ell(1+ |y |)^\ell\quad  \forall (x,y,u)\in \R^d\times\R^d\times U,
	\] 
	where $C>0$ is a constant depending on only $\ell$, $\lambda$, $\Theta$, 
	$\zeta$, and $\alpha$. 
\end{Lemma}

\begin{proof}
	Fix $x, y\in \R^d$. For an arbitrary $u_0\in U$, thanks to Assumption \ref{assume.u.lips}, 
	\begin{equation}\label{r+by Lips}
		\left|\big(r(x,u)+b(x,u)\cdot y\big) - \big(r(x,u_0)+b(x,u_0)\cdot y\big)\right| \leq \Theta (1+ |y|)|u-u_0|,\quad \forall u\in U. 
	\end{equation}
	This readily implies
	\be\label{eq.lm.ent1} 
	\begin{aligned}
		\Gamma(x,y,u_0) &= \frac{\exp\left(\frac{1}{\lambda}[r(x,u_0)+b(x,u_0)\cdot y] \right)}{\int_U \exp\left(\frac{1}{\lambda}[r(x,u)+b(x,u)\cdot y] \right)du}
		\leq \frac{1}{\int_U \exp\left(-\frac{\Theta}{\lambda}( 1+ |y|)|u-u_0|\right) du},
	\end{aligned}
	\ee
	where the inequality follows from dividing by $\exp(\frac{1}{\lambda}[r(x,u_0)+b(x,u_0)\cdot y])$ the numerator and the denominator and then using the estimate \eqref{r+by Lips}. 
	
	Now, let us prove the desired result for $\ell>1$ first. By Assumption \ref{assume.U.cone}, 
	\be\label{eq.lm.ent}  
	\begin{aligned}
		\int_U e^{-\frac{\Theta}{\lambda} (1 + |y|)|u-u_0|}  du &\geq \int_{\text{cone}(u_0,\alpha)\cap B_\zeta(u_0)} e^{-\frac{\Theta}{\lambda} (1 + |y|)|u-u_0|}  du 
		= \int_{\Delta_\alpha \cap B_\zeta(0)} e^{-\frac{\Theta}{\lambda} (1 + |y|) |u|} du,
	\end{aligned}
\ee 
	where the equality follows from a translation from $u_0$ to $0\in\R^\ell$ and a proper rotation about $0\in\R^\ell$. Let us estimate the right-hand side of \eqref{eq.lm.ent} in the next two cases. If $\frac{\Theta}{\lambda} (1+ |y|)\zeta \leq 1$, 
	$$
	\int_{\Delta_\alpha \cap B_\zeta(0)} e^{-\frac{\Theta}{\lambda} (1 + |y|) |u|} du \geq e^{-1} \Leb(\Delta_\alpha \cap B_\zeta(0)) =: K_0. 
	$$
	If $\frac{\Theta}{\lambda} (1+ |y|)\zeta > 1$, consider the two positive constants
	\[
	K_1 :=  \int_0^\pi\sin^{\ell-2}(\varphi_1)d\varphi_1 \cdots \int_0^\pi \sin (\varphi_{\ell-2})d\varphi_{\ell-2} \int_{-\alpha}^{\alpha} d \varphi_{\ell-1}\quad \hbox{and}\quad K_2 := \int_0^1z^{\ell-1}e^{-z}dz. 
	\]
	By using the $\ell$-dimensional spherical coordinates, we have
	\bee
	\begin{aligned}
		\int_{\Delta_\alpha \cap B_\zeta(0)} e^{-\frac{\Theta}{\lambda} (1 + |y|) |u|} du &= K_1 \int_0^\zeta r^{\ell-1}  e^{-\frac{\Theta}{\lambda} (1 + |y|) r} dr
		= K_1 \bigg(\frac{\lambda}{\Theta (1+ |y| )}\bigg)^\ell \int_0^{\frac{\Theta}{\lambda}(1 + |y|)\zeta} z^{\ell-1}e^{-z}dz\\
		& \geq K_1 K_2 \bigg(\frac{\lambda}{\Theta (1+ |y| )}\bigg)^\ell, 
	\end{aligned}
	\eee
	where the second equality follows from the change of variable $z=\frac{\Theta}{\lambda}(1+ |y|) r$. Combining the above two cases, we conclude from \eqref{eq.lm.ent1} and \eqref{eq.lm.ent} that
	$
	\Gamma(x,y,u_0) \leq \max\left\{\frac{1}{K_0},   \frac{1}{K_1K_2} \left(\frac{\Theta}{\lambda}(1+ |y|)\right)^\ell \right\},
	$	
	which implies the desired result as $K_0$, $K_1$, and $K_2$ depend on only $\alpha$, $\zeta$, and $\ell$. 	
	
	For $\ell=1$, by following arguments similar to the above, with $\Delta_\alpha\cap B_\zeta(0)$, $K_0$, $K_1$, and $K_2$ replaced by $[0,\zeta]$, $e^{-1}\zeta$, $1$, and $\int_0^1 e^{-z}dz=1-e^{-1}$, respectively, we can obtain the desired estimate. 
\end{proof}

Now, recall $\hat\Hc(x,y)$ defined in \eqref{eq.hat.notation} and \eqref{eq.h}, which stands for the entropy of the density function $u\mapsto \Gamma(x,y,u)$. We observe that the polynomial growth of $y\mapsto \Gamma(x,y,u)$ in Lemma~\ref{lm.1} swiftly turns into the logarithmic growth of $y\mapsto\hat\Hc(x,y)$.   

\begin{Corollary}\label{coro:H bdd}
	Suppose that Assumptions \ref{assume.u.lips} and \ref{assume.U.cone} hold. 
	Then, there exists $\kappa>0$, depending on only $\ell$, $\lambda$, $\Leb(U)$, $\Theta$, $\zeta$, and $\alpha$, such that
	\[
	\sup_{x\in \R^d} \big|\hat\Hc(x,y)\big|  \leq  \kappa +\lambda \ell \ln(1+|y|),\quad \forall y\in\R^d. 
	\]	
\end{Corollary}

\begin{proof}
	By Lemma~\ref{lm.1}, there exists a constant $C>0$, depending on only $\ell$, $\lambda$, $\Theta$, $\zeta$, and $\alpha$,   such that $\ln(\Gamma(x,y,u)) \le \ell \ln C+\ell\ln (1+|y|)$ for all $(x,y,u)\in\R^d\times\R^d\times U$. As $u\mapsto \Gamma(x,y,u)$ is by definition a density function (see \eqref{eq.gibbs.pi}), this implies
	$
	\sup_{x\in\R^d} \int_U \ln(\Gamma(x,y,u)) \Gamma(x,y,u)du \le \ell \ln C+\ell\ln (1+|y|),\ \forall y\in \R^d. 
	$ 
	On the other hand, \eqref{eq.kl.lbound} readily gives $\sup_{x\in\R^d} \int_U \ln(\Gamma(x,y,u)) \Gamma(x,y,u)du \geq -\ln(\Leb(U)).$
	Hence, in view of \eqref{eq.hat.notation} and \eqref{eq.h}, we conclude that for all $y\in\R^d$
	\begin{align*}
		\sup_{x\in \R^d} \left|\hat\Hc(x,y)\right| = \lambda	\sup_{x\in \R^d}  \left| \int_U \ln(\Gamma(x,y,u))\Gamma(x,y,u)du \right|  
		\leq \lambda \big(|\ln(\Leb(U))| +  \ell |\ln C|+\ell\ln (1+|y|)\big),\
	\end{align*}	
	which shows that the desired result holds.
\end{proof}

\begin{Remark}\label{rem:no sublinear growth}
As pointed out in Remark~\ref{rem:sublinear important}, the logarithmic growth of $\|\hat\Hc(\cdot,p(\cdot))\|_{\Cc^0(\R^d)}$ in $\|p\|_{\Cc^0(\R^d)}$ is indispensable for our eventual proof of \eqref{eq.thm.uniform}. Notably, this modest growth is achieved under the $\Cc^0(\R^d)$-norm, but not under any $\Cc^k(\R^d)$-norm with $k\ge 1$. Indeed, Step 4 in the proof of Lemma \ref{lemma.lips.xy} (i) in Appendix \ref{sec:appendix} 
indicates that $|D_{x} \hat\Hc(\cdot,p(\cdot))|$ has at least linear growth in $\|p\|_{\Cc^1(\R^d)}$. 
\end{Remark}


\subsection{Achieving Optimality}\label{subsec:optimality}
Based on the preparations in Sections~\ref{subsec:Sobolev} and \ref{subsec:entropy}, the next result carefully implements the ``grand plan'' set forth in Section~\ref{subsec:grand plan} and ultimately establishes the desired uniform bound \eqref{eq.thm.uniform}. 

\begin{Proposition}\label{thm.the.uniformbound}
	Let $\Lambda_{\lfloor d/2 \rfloor+2}<\infty$ in \eqref{eq.parameters.m} and assume \eqref{eq.assume.ellip}. Suppose additionally that Assumptions \ref{assume.u.lips} and  \ref{assume.U.cone} hold. 
	Then, the sequence $\{v^n\}_{n\in\N}$ in PIA \eqref{PIA} satisfies \eqref{eq.thm.uniform}. 
\end{Proposition}

\begin{proof}
	For convenience, we take $s :=\lfloor d/2 \rfloor+1$ throughout the proof. 
	
	Step 1. For any $x\in \R^d$ and $n\geq s+2$, we will show that $|D_x v^n(x)|+|D^2_x v^n(x)|$ is bounded from above by the $L^q$- and $W^{k,q}$-norms of $\{v^m\}_{m=n-s}^{n}$. By Corollary \ref{prop.vn.bound} (with $\bar k = s+1$), $v^{n}\in  \Cc^{s+2,\alpha}_{\text{unif}}(\R^d) \subset \Cc^{s+2}(\R^d)$. With $D_x v^n, D_x^2v^{n}\in \Cc^{s}(\R^d)$ and $s>d/2$, Lemma \ref{lm.infinity.emded} implies
	\be\label{eq.infinity.hnorm}
	\begin{aligned}
		|D_x v^n(x)|+|D^2_x v^n(x)|\leq C  \|v^n\|_{W^{s+2,2}(B_1(x))}, 
	\end{aligned}
	\ee
	where $C>0$ is a constant that depends on only $s$ and $d$. 
	
	As Corollary \ref{prop.vn.bound} (with $\bar k = s+1$) also gives $v^{m}\in\Cc^{(m+1)\wedge (s+2)}(\R^d)$ for all $1\le m\le n$, by recalling the notation in \eqref{eq.hat.notation} and \eqref{eq.h}, for all $2\le m\leq n$ we define
	\begin{equation}\label{fmbm}
		f^m(\cdot):= \hat{r}(\cdot , D_x v^{m-1}(\cdot))-\hat{\Hc}(\cdot,D_x v^{m-1}(\cdot)),\quad  b^m(\cdot):=\hat{b}(\cdot,D_x v^{m-1}(\cdot)).
	\end{equation}
	For each $2\le m\le n$, as $D_x v^{m-1}\in  \Cc^{(m-1)\wedge (s+1)}(\R^d)$, by Lemma~\ref{lemma.lips.xy} (i) (with $k=(m-1)\wedge (s+1)$ and $p=D_x v^{m-1}$), we have $b^m\in \Cc^{(m-1)\wedge (s+1)}(\R^d)$. Since $v^m$ satisfies \eqref{v^n-1 to v^n} (recall Corollary \ref{prop.vn.bound}), the conditions $b^m\in \Cc^{(m-1)\wedge (s+1)}(\R^d)$, $\sigma\in\Cc^{s+1}(\R^d)$ (as $\Lambda_{\lfloor d/2 \rfloor+2}<\infty$), and \eqref{eq.assume.ellip} allow us to apply Lemma \ref{lm.hnorm.estimate} to $v^m$. This implies that for any $k\leq (m-1)\wedge (s+1)$ and $(q,\eta)\in [1,\infty)\times \R_+$, 
	\be\label{eq.vn.fn}  
	\begin{aligned}
		\|v^m\|_{W^{k+2,q}(B_\eta(x))} &\leq  H_k\left(1+\|b^m\|_{W^{k,A(k,q)}(B_{R(k,\eta)}(x))}\right)^{Q(k,q)}\\
		&\hspace{0.2in}\cdot \left(\|v^m\|_{L^{A(k,q)}(B_{R(k,\eta)}(x))} +\|f^m\|_{W^{k,A(k,q)}(B_{R(k,\eta)}(x))}\right),
	\end{aligned}
	\ee
	where $H_k>0$ is a constant depending on only $k$, $\Lambda_k$, $d$, $q$, $\eta$, and $\rho$. Furthermore, by taking
	\begin{equation}\label{bar values}
		\bar q := A(k,q)\ge q\quad \hbox{and}\quad \bar \eta := R(k,\eta)\ge 2\eta,
	\end{equation}
	we deduce from Lemma \ref{lm.fn.vn-1} (with $p=D_x v^{m-1}$) that 
	\be\label{eq.fn.vn-1}  
	\max\left\{\|f^m\|_{W^{k,\bar q}(B_{\bar\eta}(x))},\ \|b^m\|_{W^{k,\bar q}(B_{\bar\eta}(x))}\right\}  \leq C_{k} \left(1+\|v^{m-1}\|_{W^{k+1,a_k \bar q }(B_{\bar \eta}(x))}\right)^{a_k},
	\ee
	where we take $a_k:=  (k+1)2^{k+1}$ and $C_k>0$ is a constant depending on only $k, \Lambda_k, \lambda, d, q$, and $\eta$. By plugging \eqref{eq.fn.vn-1} into \eqref{eq.vn.fn}, we obtain
	\be\label{eq.iter.n} 
	\begin{aligned}
\|v^m\|_{W^{k+2,q}(B_\eta(x))} 
	&\leq {H}_k \big(1+C^{1/{a_k}}_k\big)^{a_k Q(k,q)}\Big(1+\|v^{m-1}\|_{W^{k+1,a_kA(k,q) }(B_{R(k,\eta)}(x))}\Big)^{a_kQ(k,q)} \\
		&\hspace{0.15in}\cdot \big(1+C_k^{1/{a_k}}\big)^{a_k}\Big(1+\|v^m\|_{L^{A(k,q)}(B_{R(k,\eta)}(x))} +\|v^{m-1}\|_{W^{k+1,a_kA(k,q) }(B_{R(k,\eta)}(x))}\Big)^{a_k} \\
		&\hspace{-0.2in}\leq \widetilde{H}_k\Big(1+\|v^m\|_{L^{A(k,q)}(B_{R(k,\eta)}(x))} +\|v^{m-1}\|_{W^{k+1,a_kA(k,q) }(B_{R(k,\eta)}(x))}\Big)^{a_k(Q(k,q)+1)}, 
	\end{aligned}
\ee 
	where $\widetilde{H}_k:= {H}_k(1+C^{1/{a_k}}_k)^{a_k(Q(k,q)+1)}>0$ is finite and depends on only $k$, $\Lambda_k$, $\lambda$, $d$, $q$, $\eta$, and $\rho$. 
	
	Now, set $(k_0,q_0,\eta_0) := (s,2,1)$ and define $\{(k_i,q_i,\eta_i,\theta_i)\}_{i=1}^s$ by 
	\begin{align*}
		&k_i := k_{i-1}-1,\qquad q_i := a_{k_{i-1}} A(k_{i-1},q_{i-1})\ge A(k_{i-1},q_{i-1}),\\ 
		&\eta_i := R(k_{i-1},\eta_{i-1})\ge 2\eta_{i-1},\qquad \theta_i := a_{k_{i-1}} (Q(k_{i-1},q_{i-1})+1), 
	\end{align*}
	for all $i=1,\ldots,s$. It then follows from \eqref{eq.infinity.hnorm} and \eqref{eq.iter.n} that 
	\begin{align*}
		&|D_x v^n(x)|+|D^2_x v^n(x)|
		\le C  \|v^n\|_{W^{s+2,2}(B_1(x))}\notag
		\le C \widetilde{H}_s\Big(1+\|v^n\|_{L^{q_1}(B_{\eta_1}(x))} +\|v^{n-1}\|_{W^{s+1,q_1 }(B_{\eta_1}(x))}\Big)^{\theta_1}\\
		&\le C \widetilde{H}_s\big(1+ (\widetilde H_{s-1})^{1/\theta_2}\big)^{\theta_1\theta_2}
		\cdot \Big(1+\|v^n\|_{L^{q_1}(B_{\eta_1}(x))} + \|v^{n-1}\|_{L^{q_2}(B_{\eta_2}(x))} + \|v^{n-2}\|_{W^{s,q_2 }(B_{\eta_2}(x))}\Big)^{\theta_1\theta_2}, 
	\end{align*}
	where the second inequality follows by taking $(m,k,q,\eta) = (n,k_0,q_0,\eta_0)$ in \eqref{eq.iter.n}, while the third inequality holds similarly by taking $(m,k,q,\eta) = (n-1,k_1,q_1,\eta_1)$ in \eqref{eq.iter.n}. By continuing this procedure (i.e., taking $(m,k,q,\eta) = (n-i,k_i,q_i,\eta_i)$ in \eqref{eq.iter.n} for all $i=2,\ldots,s-1$), we obtain
	\be\label{eq.vn.h0norm0} 
	\begin{aligned}
		&|D_x v^n(x)|+|D^2_x v^n(x)|
		\le H^*\bigg(1+\sum_{i=0}^{s-1} \|v^{n-i}\|_{L^{q_{i+1}}(B_{\eta_{i+1}}(x))} + \|v^{n-s}\|_{W^{2,q_{s}}(B_{\eta_{s}}(x))}\bigg)^{\theta_1\theta_2\ldots\theta_{s}},
	\end{aligned}
\ee 
	where $H^* := C \widetilde H_s \Pi_{i=1}^{s-1}\big(1+(\widetilde H_{s-i})^{1/\theta_{i+1}}\big)^{\theta_1\ldots\theta_{i+1}}>0$ is finite and depends on only $s$, $\Lambda_s$, $\lambda$, $d$, and $\rho$. 
	
	Step 2. We will show that the right-hand side of \eqref{eq.vn.h0norm0} can be made independent of $x\in\R^d$. By Lemma~\ref{lem:V^pi bdd} and Corollary~\ref{coro.PIA.works},  
	\be\label{eq.thm1.0}  
	\begin{aligned}
	\|v^n\|_{\Cc^0(\R^d)} & \leq   \max\{\|v^1\|_{\Cc^0(\R^d)},\big(\|r\|_{\Cc^0(\R^d)}+\lambda|\ln(\Leb (U))|\big)/\rho\}=: M<\infty,\quad \forall n\in\N.
	\end{aligned}
	\ee
	Then, by taking $(m,k,q,\eta) = (n-s,0,q_s,\eta_s)$ in \eqref{eq.vn.fn} and recalling \eqref{fmbm}, we get
	\be\label{eq.vn.h0norm0'''}
	\begin{aligned}
		\| v^{n-s}\|_{ W^{2,q_{s}}(B_{\eta_{s}}(x)) }
		 &\leq {H}_0 \Big(1+\|b^{n-s}\|_{L^{A(0, q_{s})}(B_{R(0, \eta_{s})}(x))} \Big)^{Q(0,q_{s})} \\
		&\quad \cdot \Big( \| v^{n-s}\| _{ L^{A(0, q_{s})}(B_{R(0,\eta_{s})}(x))} + \|  f^{n-s}\|_{L^{A(0, q_{s})}(B_{R(0,\eta_{s})}(x))} \Big) \\
		&\hspace{-0.2in}\leq {H}_0 \Big(1+\|b\|_{\Cc^0(\R^d)}\big(\Leb(B_{R(0,\eta_{s})}(x))\big)^{\frac{1}{A(0, q_{s})}} \Big)^{Q(0, q_{s})}  \\
		&\hspace{-0.2in}\quad \cdot \big(\Leb(B_{R(0,\eta_{s})}(x))\big)^{\frac{1}{A(0, q_{s})}}\Big(M + \|r\|_{\Cc^0(\R^d)} +\|\hat{\Hc}(\cdot,D_x v^{n-s-1}(\cdot))\|_{\Cc^0(\R^d)}  \Big). 
	\end{aligned}
	\ee 
	As $\Leb(B_{R(0,\eta_{s})}(x))$ is a constant independent of $x\in\R^d$, we deduce from Corollary~\ref{coro:H bdd} that the above inequality can be turned into
	$
	\| v^{n-s}\|_{ W^{2,q_{s}}(B_{\eta_{s}}(x)) } \le  C^* \big(1 + \ln\big(1+\|D_x v^{n-s-1}\|_{\Cc^0(\R^d)}\big)\big), 
	$
	where $C^*>0$ is finite and depends on only $s$, $\Lambda_s$, $\lambda$, $d$, $\rho$, $\ell$, $\Leb(U)$, $\Theta$, $\zeta$, and $\alpha$. Thus, \eqref{eq.vn.h0norm0}  becomes
	\be\label{eq.key}
	\begin{aligned}
		&|D_x v^n(x)|+|D^2_x v^n(x)| \\
		&\leq H^*\bigg(1+M \sum_{i=1}^{s} \Leb(B_{\eta_{i}}(x))^{1/q_{i}}
		 + C^* \left(1 + \ln\Big(1+\|D_x v^{n-s-1}\|_{\Cc^0(\R^d)}\Big)\right) \bigg)^{\theta_1\theta_2\ldots\theta_{s}} \\
		&\leq \widetilde C^*\left(1 + \ln\Big(1+\|D_x v^{n-s-1}\|_{\Cc^0(\R^d)}\Big)\right)^{\theta_1\theta_2\ldots\theta_{s}},
	\end{aligned}
\ee 
	where $\widetilde C^*>0$ is a constant depending on only $s$, $\Lambda_s$, $\lambda$, $d$, $\rho$, $\ell$, $\Leb(U)$, $\Theta$, $\zeta$, and $\alpha$.
	
	Step 3. We are now ready to prove the desired result \eqref{eq.thm.uniform}. As we already know from \eqref{eq.thm1.0} that $\sup_{n\in\N}\|v^n\|_{\Cc^0(\R^d)}<\infty$, it suffices to show $\sup_{n\in\N}\|D_x v^n\|_{\Cc^1(\R^d)}<\infty$. 
	By the arbitrariness of $x\in\R^d$ in \eqref{eq.key} and the fact $\|\cdot\|_{\Cc^0(\R^d)}\le \|\cdot\|_{\Cc^1(\R^d)}$, we have 
	\begin{align}\label{eq.thm.key} 
		\|D_x v^n\|_{\Cc^1(\R^d)}&\leq \widetilde C^*\left(1 + \ln\Big(1+\|D_x v^{n-s-1}\|_{\Cc^1(\R^d)}\Big)\right)^{\theta_1\theta_2\ldots\theta_{s}},\quad \forall n\ge s+2.
	\end{align}
	Note that $\phi(z) := \widetilde C^*\left(1 + \ln(1+z)\right)^{\theta_1\theta_2\ldots\theta_{s}}$, for $z\ge 0$, satisfies $\phi(0)=\widetilde C^*>0$ and is strictly increasing and sublinear (i.e., $\phi(z)/z\to 0$ as $z\to\infty$). Hence,
	$
	z_0:=\sup\{ z\ge 0: z\leq \phi(z) \} <\infty.
	$
	Consider
$	z^*:=\max\left\{\max_{1\leq n\leq s+1}\|D_x v^n\|_{\Cc^1(\R^d)},\ z_0\right\}<\infty.$
	By definition, $\|D_x v^n\|_{\Cc^1(\R^d)}\leq z^*$ for $n=1,\ldots,s+1$. Moreover, by \eqref{eq.thm.key}, 
	\be\label{s+2 to 2(s+1)} 
	\begin{aligned}
		\max_{s+2\leq n\leq 2(s+1)}\|D_x v^n\|_{\Cc^1(\R^d)} &\leq  \widetilde C^*\left(1 + \ln\Big(1+\max_{1\leq n\leq s+1}\|D_x v^{n}\|_{\Cc^1(\R^d)}\Big)\right)^{\theta_1\theta_2\ldots\theta_{s}} \\
		&= \phi\Big(\max_{1\leq n\leq s+1}\|D_x v^{n}\|_{\Cc^1(\R^d)}\Big) \leq \phi(z^*)\le z^*,
	\end{aligned}
\ee 
	where the last inequality follows from $z^*\ge z_0$ and the definition of $z_0$. That is, 
	$\|D_x v^n\|_{\Cc^1(\R^d)}\leq z^*$ $\text{for}\ n=s+2,\ldots,2(s+1).$
	By repeating the argument in \eqref{s+2 to 2(s+1)}, we can show 
	$
	\|D_x v^n\|_{\Cc^1(\R^d)} \le z^*\ \hbox{for}\ n= j(s+1)+1,\ldots, (j+1)(s+1),\ \hbox{with}\ j=2,3,\ldots .
	$
	We then conclude $\sup_{n\in\N}\|D_x v^n\|_{\Cc^1(\R^d)}\le z^*<\infty$.
\end{proof}

As pointed out in Remark~\ref{rem:detour}, the ``grand plan'' in Section~\ref{subsec:grand plan}, followed closely by the proof above, seems an overly complicated detour. Such a detour might be necessary as it turns out.

\begin{Remark}\label{rem:detour necessary}
	It is tempting to prove $\sup_{n\in\N} \| D_x v^n\|_{\Cc^1(\R^d)}<\infty$ directly in the space $\Cc^1(\R^d)$. By using a Schauder estimate (precisely for the second inequality below), one gets
	\begin{align}\label{eq.sc.try0} 
		\|D_x v^n\|_{\Cc^{1}(\R^d)}\le \|D_x v^n\|_{\Cc^{1, \alpha}_{\text{unif}}(\R^d)} &\leq K_n \left(1+\|\hat{\Hc}(\cdot, D_x v^{n-1}(\cdot))\|_{\Cc^{0, \alpha}_{\text{unif}}(\R^d)}\right)\notag\\
		&\le K_n \left(1+\|\hat{\Hc}(\cdot, D_x v^{n-1}(\cdot))\|_{\Cc^{1}(\R^d)}\right),
	\end{align} 
	where $K_n>0$, the usual constant in a Schauder estimate, shall take the form $K_n = g(\|D_x v^{n-1} \|_{C^{0, \alpha}_{\text{unif}}(\R^d)})$ for some increasing $g:\R_+\to \R_+$, in the present context of policy iteration. 
	Ideally, if $\{K_n\}_{n\in\N}$ has a uniform upper bound and $\|\hat{\Hc}(\cdot, D_x v^{n-1}(\cdot) )\|_{\Cc^{1}(\R^d)}$ grows sublinearly in $\|D_x v^{n-1}\|_{\Cc^{1}(\R^d)}$, the right-hand side of \eqref{eq.sc.try0} will grow sublinearly in $\|D_x v^{n-1}\|_{\Cc^{1}(\R^d)}$, similarly to \eqref{eq.thm.key}. The estimates below \eqref{eq.thm.key} can then be applied to get $\sup_{n\in\N} \| D_x v^n\|_{\Cc^1(\R^d)}<\infty$. This argument, however, breaks down because $\|\hat{\Hc}(\cdot, D_x v^{n-1}(\cdot) )\|_{\Cc^{1}(\R^d)}$ grows at least linearly in $\|D_x v^{n-1}\|_{\Cc^{1}(\R^d)}$ (see  Remark~\ref{rem:no sublinear growth}).

	That is to say, the true advantage of our ``grand plan'' is that it allows us to control $\|D_x v^n\|_{\Cc^1(\R^d)}$ by the $\Cc^0(\R^d)$-norm of $\hat{\Hc}(\cdot,D_x v^{n-s-1}(\cdot))$, instead of its $\Cc^1(\R^d)$-norm (see \eqref{eq.vn.h0norm0} and \eqref{eq.vn.h0norm0'''}),  so that the crucial sublinear growth can be imported from Corollary~\ref{coro:H bdd}. 
\end{Remark}

Now, we are ready to present the main result of this paper: with appropriate regularity on $r$, $b$, $\sigma$, and the action space $U$, the PIA \eqref{PIA} yields the optimal value and an optimal relaxed control. 

\begin{Theorem}\label{thm.verification.limit}
	Let $\Lambda_{\lfloor d/2 \rfloor+2}<\infty$ in \eqref{eq.parameters.m} and assume \eqref{eq.assume.ellip}. Suppose additionally that Assumptions \ref{assume.u.lips} and \ref{assume.U.cone} hold. Then, the limiting function $v^*$ in the PIA \eqref{PIA}, defined as in \eqref{v^*}, satisfies
	\begin{enumerate}[\normalfont(i)]
\item $v^* = V^*$ on $\R^d$ and $\pi^*(x,u):=\Gamma(x,D_x v^*(x),u)\in\Ac_M$ is an optimal relaxed control for \eqref{V^*};
\item $v^*$ is the unique solution in $ \Cc^2(\R^d)$ to the HJB equation \eqref{eq.HJB.new}; 
\item
		$v^*\in \Cc^{\lfloor d/2 \rfloor+3,\alpha}_{\text{unif}}(\R^d)$ for all $0<\alpha<1$.
	\end{enumerate}
\end{Theorem}

\begin{proof}
	Consider $\{v^n\}_{n\in\N}$ in the PIA \eqref{PIA} and recall from Corollary~\ref{coro.PIA.works} that its limit $v^*$ is well-defined. For any compact subset $E$ of $\R^d$, Proposition~\ref{thm.the.uniformbound} implies that $\{v^n\}_{n\in\N}$ (resp.\ $\{D_x v^n\}_{n\in\N}$) is uniformly bounded and equiconstinuous. By the Arzela--Ascoli theorem, $v^n$ (resp.\ $D_x v^n$) converges uniformly on $E$, up to a subsequence. 
We then conclude $v^*\in \Cc^1(E)$ and $D_x v^n\to D_x v^*$. Recall from Corollary~\ref{prop.vn.bound} that $v^n$ satisfies \eqref{v^n-1 to v^n} for all $n\in\N$. By the uniform convergence of $v^n$ and $D_x v^n$ on $E$, as $n\to\infty$ in \eqref{v^n-1 to v^n}, we observe that $D_x^2 v^n$ must also converge uniformly on $E$. This in turn implies
 $D_x v^*\in \Cc^1(E)$ and $D^2_x v^n\to D^2_x v^*$. Hence, $v^*\in \Cc^2(E)$ and the limiting equation of \eqref{v^n-1 to v^n} on $E$ (as $n\to\infty$) is
	\begin{equation}\label{v^* eqn}
		\Lc^{\pi^*} v^*(x)+\hat{r}(x, D_x v^{*}(x))-\hat{\Hc}(x, D_x v^{*}(x))=0. 
	\end{equation}
	By the definitions of $\Lc^{\pi^*}$, $\hat r$, $\hat\Hc$, and $\Gamma$ (see \eqref{eq.elliptic.gibbs}, \eqref{eq.hat.notation}, \eqref{eq.h}, and \eqref{eq.gibbs.pi}), this readily shows that $v^*$ satisfies the HJB equation \eqref{eq.HJB.new} on $E$. 
	As $E\subset\R^d$ is arbitrary, we conclude that $v^*\in\Cc^2(\R^d)$ and it satisfies \eqref{eq.HJB.new} on $\R^d$. We then have $\pi^*\in\Ac_M$ by Proposition \ref{prop.sde}. Also, a modified verification argument (for tackling the entropy term) shows $v^*= V^*$ on $\R^d$. To see this, for any $x\in\R^d$ and $\pi\in\Ac$, by It\^{o}'s formula (applied to $v^*(X^\pi)$) and 
	the boundedness of $D_x v^*$ and $\sigma$, it holds for all $T>0$ that
	\begin{align}
		v^*(x) =&\E\bigg[e^{-\rho T}v^*(X^{\pi}_T)  \notag 
		+ \int_0^T e^{-\rho t} \bigg\{\rho v^*-  \left(\int_U b(\cdot,u)\pi_t(u)du\right)\cdot D_x v^*
		-\frac{1}{2}\tr(\sigma \sigma' D^2_x v^*)\bigg\} (X^{\pi}_t) dt \bigg] \notag \\
		\geq& \E[e^{-\rho T}v^*(X^{\pi}_T)] 
		+ \E\bigg[\int_0^T e^{-\rho t} \left( \int_U r(X_t^{\pi},u)\pi_t(u) du-\lambda \int_U \pi_t(u)\ln \pi_t(u)du \right) dt \bigg] \label{verifi} \\
		= & \E[e^{-\rho T}v^*(X^{\pi}_T)]+ \E\bigg[\int_0^T e^{-\rho t} \left( \int_U r(X_t^{\pi},u)\pi_t(u) du\right) dt\bigg] \notag  \\
		&- \lambda\E\bigg[\int_0^\infty e^{-\rho t} \left(\int_U \pi_t(u)\ln \pi_t(u)du \right) dt \bigg] 
		+\lambda\E\bigg[\int_T^\infty e^{-\rho t} \left(\int_U \pi_t(u)\ln \pi_t(u)du \right) dt \bigg] \notag 
	\end{align}
	where the inequality holds as $v^*$ satisfies the HJB equation \eqref{eq.HJB.new}. As $T\to\infty$, since $v^*$ and $r$ are bounded and $\int_U \pi_t(u)\ln \pi_t(u)du$ is bounded from below by $-\ln(\Leb(U))$ (recall \eqref{eq.kl.lbound}), we deduce from the dominated convergence theorem and Fatou's lemma that
	\[
	v^*(x) \ge \E\bigg[\int_0^\infty e^{-\rho t} \left( \int_U r(X_t^{\pi},u)\pi_t(u) du-\lambda \int_U \pi_t(u)\ln \pi_t(u)du \right) dt \bigg]=V^\pi(x).
	\]
	Taking supremum over $\pi\in\Ac$ yields $v^*(x)\ge V^*(x)$. When we take $\pi=\pi^*\in\Ac_M$ in \eqref{verifi}, the inequality becomes an equality as $\pi^*(x,u)= \Gamma(x,D_x v^*(x),u)$ attains the supremum in \eqref{eq.HJB.new} (with $v$ taken to be $v^*$), thanks to \eqref{eq.gibbs.pi}. This gives 	
	$
	v^*(x) = \E[e^{-\rho T}v^*(X^{\pi^*}_T)]+ \E\big[\int_0^T e^{-\rho t} (\hat r- \hat\Hc)(X_t^{\pi^*}, D_xv^*(X_t^{\pi^*}) ) dt \big]. 
	$
	As $T\to\infty$, since $v^*$ and $(\hat r- \hat\Hc)(\cdot, D_x v^*(\cdot))$ are bounded (by Lemma~\ref{lemma.lips.xy} (i) with $k=0$), the dominated convergence theorem implies $v^*(x) = V^{\pi^*}(x)\le V^*(x)$. We then conclude $v^*=V^*$ on $\R^d$ and $\pi^*$ is an optimal relaxed control for \eqref{V^*}. For any solution $w\in \Cc^2(\R^d)$ to \eqref{eq.HJB.new}, the same arguments as above (with $w$ in place of $v^*$) entail $w=V^*$. Hence, $v^*=V^*$ is the unique solution in $\Cc^2(\R^d)$ to \eqref{eq.HJB.new}.
	
	As $v^*\in\Cc^2(\R^d)\subset \Cc^{1,1}_{\text{unif}}(\R^d)$, we take $v^0 = v^*$ in the PIA \eqref{PIA}, then  $v^n = v^*$ for all $n\in\N$. With $\Lambda_{\lfloor d/2 \rfloor+2}<\infty$,  Corollary~\ref{prop.vn.bound} readily implies $v^*\in \Cc^{\lfloor d/2 \rfloor+3,\alpha}_{\text{unif}}(\R^d)$ for all $0<\alpha<1$.
\end{proof}

\begin{Remark}\label{rem:complements}
	At the end of \cite[Section 6]{tang2021exploratory}, an open question is raised: under what conditions on $r$, $b$, and $\sigma$ does the (exploratory) HJB equation \eqref{eq.HJB.new} have a unique solution of the class $\Cc^3(\R^d)$? Theorem~\ref{thm.verification.limit} provides an answer to this: 
	Under the conditions that (i) $r(\cdot,u)$, $b(\cdot, u)$, and $\sigma(\cdot, u)$ belong to $\Cc^{\lfloor d/2 \rfloor+2}(\R^d)$, (ii) $r(x,\cdot)$ and $b(x,\cdot)$ are Lipschitz, (iii) $\sigma$ is uniformly elliptic, and (iv) the action space $U$ satisfies a uniform cone condition, $V^*$ in \eqref{V^*} is characterized as the unique classical solution to \eqref{eq.HJB.new} with desirable regularity, i.e., $V^*\in \Cc^{\lfloor d/2 \rfloor+3,\alpha}_{\text{unif}}(\R^d)\subset \Cc^3(\R^d)$. 
\end{Remark}

\subsection{Discussion: The Case of a Controlled Diffusion Coefficient}\label{subsec:only drift controlled}
When the diffusion coefficient of \eqref{eq.sde.new} is also controlled, i.e., 
$$ 
dX^\pi_t= \left( \int_U b(X^\pi_t,u)\pi_t(u)du \right) dt + \bigg( \int_U (\sigma \sigma')(X_t^\pi,u)\pi_t(u)du\bigg)^{1/2} \cdot dW_t,
$$
the Gibbs-form function $\Gamma$ in \eqref{eq.gibbs.pi} should take the following form: for any $(x,y,z,u)\in \R^d\times \R^d \times S^d\times U$, where $S^d$ denotes the set of all $d\times d$ symmetric matrices, 
\be\label{Gibbs general}
\Gamma(x,y,z,u):= \dfrac{\exp(\frac{1}{\lambda}[b(x,u)\cdot y +\frac{1}{2}\tr(\sigma\sigma' (x,u)z)+r(x,u)] )}{\int_U \exp(\frac{1}{\lambda}[b(x,u)\cdot y +\frac{1}{2}\tr(\sigma\sigma' (x,u)z)+r(x,u)] ) du}. 
\ee 
With sufficient regularity of $b$, $r$, and $\sigma$, a policy improvement result similar to Proposition~\ref{prop.policy.improvement} still holds; see e.g., \cite[Theorem 1]{JZ22}. The convergence of the PIA (i.e., \eqref{PIA} with $\Gamma(x, D_x v^{n-1}(x),u)$ replaced by $\Gamma(x, D_x v^{n-1}(x), D^2_x v^{n-1}(x), u)$) is nonetheless much more involved and remains largely open. As explained below, our ``grand plan'' does not easily apply to the present case. 

To obtain a uniform estimate for $\{v^n\}_{n\in\N}$ in the PIA \eqref{PIA}, the proof of Proposition~\ref{thm.the.uniformbound} above relies on the ``order reduction'' in  \eqref{eq.iter.n}: for $m\in\N$ large, the $W^{k+2,q}$-norm of $v^m$ is controlled by the $W^{k+1,q_1}$-norm of $v^{m-1}$, with $q_1\ge q$. By repeating this procedure, we ultimately control the $W^{k+2,q}$-norm of $v^m$ by the $W^{2,q_2}$-norm of $v^{m-k}$, with $q_2\ge q_1$. As shown in \eqref{eq.vn.h0norm0'''}, this $W^{2,q_2}$-norm of $v^{m-k}$ 
can be expressed in terms of the $\Cc^0(\R^d)$-norm of the entropy term $\hat\Hc$, 
which is properly controlled in Section~\ref{subsec:entropy}. This then provides a uniform estimate for $v^m$, with $m\in\N$ large enough.  

With a controlled diffusion coefficient, such ``order reduction'' no longer holds. As $\Gamma$ depends additionally on $z\in S^d$, 
$f^m$ and $b^m$ in \eqref{fmbm} now also depend on $D^2_x v^{n-1}(\cdot)$. The same estimation as in Lemma \ref{lm.fn.vn-1} then shows that the $W^{k,q}$-norms of $f^m$ and $b^m$ are controlled by the $W^{k,q'}$-norms of not only $D_x v^{n-1}$ but also $D^2_x v^{n-1}$, with $q'\ge q$. Thus, on the right-hand side of \eqref{eq.fn.vn-1}, the $W^{k+1,q'}$-norm of $v^{m-1}$ should now be the $W^{k+2,q'}$-norm. When plugging this into \eqref{eq.vn.fn}, we obtain an updated version of \eqref{eq.iter.n}, with the $W^{k+1,q'}$-norm of $v^{m-1}$ on the right-hand side replaced by its $W^{k+2,q'}$-norm. There is then {\it no} order reduction: on both sides of the updated \eqref{eq.iter.n}, we see a Sobolev norm of order $k+2$. 
It is then unclear how Step 2 of the ``grand plan'' (see \eqref{L^q terms}) can be achieved.


\section{An Example of Optimal Consumption}\label{sec:example}
Take $d=\bar d=\ell=1$. Consider a risky asset $S$ given by $dS_t/S_t=(\mathfrak{r}+{\mu})dt+ a dW_t$, where $\mathfrak r>0$, ${\mu}\geq 0$, and $a>0$ are the riskfree rate, the risk premium, and the volatility constant, respectively. 
An agent invests a fixed proportion $0<\eta <1$ of his wealth in $S$. 
Then, he chooses to consume a proportion $u\in U:= [\underline c,1-\eta]$ of his wealth following a relaxed control $\pi\in\Ac$, where $\underline c>0$ is the subsistence consumption level. The resulting wealth process is 
\[
dZ^{\pi}_t = \left(\int_{U} (\mathfrak{r}+{\mu}{\eta}-u) Z^\pi_t \pi_t(u)du\right) dt +a {\eta} Z^\pi_t dW_t, \quad Z_0=z\in(0,\infty).
\]
The agent measures utility from consumption by an utility function $\mathfrak u:[0,\infty)\to\R$, assumed to be strictly increasing and concave.  
The agent's goal is to achieve the optimal value 
\begin{equation}\label{frak V^*}
	\mathfrak V^*(z):=\sup_{\pi\in\Ac}\E_z\left[ \int_0^\infty e^{-\rho t} \left(  \int_U \mathfrak u (uZ^\pi_t)\pi_t(u)du-\lambda \int_U \ln(\pi_t(u))\pi_t(u) du \right) dt \right]
\end{equation}
by choosing an optimal $\pi^*\in\Ac$.  
By It\^{o}'s formula, the process $X^\pi_t:= \ln(Z^\pi_t)$, with $X_0=x:= \ln z\in \R$, satisfies \eqref{eq.sde.new} with $b(y,u) = b(u) := \mathfrak{r}+{\mu}{\eta} - \frac12 a^2\eta^2-u$ and $\sigma(y)\equiv a{\eta}$. Then, $\mathfrak V^*(z)$ can be expressed as $V^*(x)$ in \eqref{V^*}, 
with $r(y,u) := \mathfrak u(u e^y)$. 

In the following, we consider three cases of the utility function: (i) $\mathfrak u(y) := -e^{-\alpha y}$, (ii) $\mathfrak u(y) := 1/(1+e^{-\alpha y})$, and (iii) $\mathfrak u(y) := \arctan(\alpha y)$, for a fixed $\alpha> 0$.\footnote{While $1/(1+e^{-\alpha y})$ (the sigmoid function) and $\arctan(\alpha y)$ are $S$-shaped on $\R$, they are concave on $[0,\infty)$, the domain we focus on.} Let us  check the conditions of Theorem~\ref{thm.verification.limit}. As $\sigma>0$ and $U=[\underline c,1-\eta]$, \eqref{eq.assume.ellip} and Assumption~\ref{assume.U.cone} hold trivially. In Case (i), with $r(y,u) = -\exp({-\alpha u e^{y}})$, a direct calculation shows $r_y(y,u)= \alpha e^{-\alpha ue^y}ue^y$ and $r_{yy}(y,u)=\alpha e^{-\alpha ue^y} (-\alpha u^2e^{2y}+ue^y)$, which implies $|r(y,u)|\leq 1$, $|r_y(y,u)|\leq e^{-1}$ and $|r_{yy}(y,u)|\leq 4 e^{-2} + e^{-1}$. Hence, $\sup_{u\in U}\|r(\cdot,u)\|_{\Cc^{2}(\R)}< \infty$. In Cases (ii) and (iii), with $r(y,u) = 1/(1+\exp({-\alpha u e^{y}}))$ and $r(y,u) = \arctan({\alpha u e^{y}})$, direct calculations similarly imply $\sup_{u\in U}\|r(\cdot,u)\|_{\Cc^{2}(\R)}< \infty$.
As $\sup_{u\in U}\|b(\cdot,u)\|_{\Cc^{2}(\R)}= \sup_{u\in [0,1-\eta]}\|b(\cdot,u)\|_{\Cc^{0}(\R)} = \mathfrak{r}+{\mu}{\eta} +a^2 {\eta}^2/2 + 1-\eta <\infty$ and $\sup_{u\in U}\|\sigma(\cdot)\|_{\Cc^{2}(\R)}=a {\eta}<\infty$, we conclude $\Lambda_2=\Lambda_{\lfloor d/2 \rfloor+2}<\infty$ for $d=1$. Finally, to verify Assumption~\ref{assume.u.lips}, by the definition of $b$, it suffices to show that $r$ is Lipschitz in $u$ uniformly in $y$. In Case (i), with $r_u(y,u) = \alpha e^{-\alpha ue^y}e^y = r_y(y,u)/u$, the estimate above for $r_y(y,u)$ gives $|r_u(y,u)|\leq e^{-1}/u \le e^{-1}/\underline{c}$, as desired. A uniform bound for $|r_u(y,u)|$ can be similarly derived in Cases (ii) and (iii). 

As all conditions of Theorem~\ref{thm.verification.limit} are fulfilled, the PIA \eqref{PIA} can be used to find the optimal value function $V^*$ and the corresponding optimal relaxed consumption strategy $\pi^*\in\Ac$.  
Under the parameters $\underline c=\eta = 0.1$, $a=10$, $\mathfrak{r}+{\mu}{\eta} - \frac12 a^2\eta^2=\alpha=\lambda=1$ and $\rho=0.1$, the PIA \eqref{PIA}  takes the following concrete form: 
\begin{itemize}
\item[1.] Take an initial input $v^0$ and set $\pi^1(x,u):=\frac{\exp( (1-u)D_x v^0+\mathfrak u(ue^x) )}{\int_{0.1}^{0.9}\exp( (1-\tilde u)D_x v^0+\mathfrak u(\tilde ue^x) )d\tilde u}$. 
\item[2.] For $n=1,2,...$, perform
\begin{itemize}
\item {\bf Policy evaluation:} Compute $v^n=V^{\pi^n}$ by solving PDE \eqref{v^n-1 to v^n}, which is now 
\be\label{eq:eg.pdevn}  
\begin{split}
-0.1 v^n&+{D_x^2 v^n}/{2}\\
&\hspace{-0.3in}+ \int_{0.1}^{0.9}\left((1-u) D_x v^n(x)+\mathfrak u(ue^x)- \ln(\pi^n(x,u))\right) \pi^n(x,u)du = 0.
\end{split}
\ee 
\item{\bf Policy improvement:} Set $\pi^{n+1}(x,u)=\frac{\exp( (1-u)D_x v^n+\mathfrak u(ue^x) )}{\int_{0.1}^{0.9}\exp( (1-\tilde u)D_x v^n+\mathfrak u(\tilde ue^x) )d\tilde u}$.
\end{itemize}
\end{itemize}
When performing this algorithm for each of the three cases of $\mathfrak u$, we take two different initial guesses, $v^0(x)=\sin(x)$ and $v^0(x)=\frac{1}{1+x^2}$, and solve \eqref{eq:eg.pdevn} using finite difference methods. 
In addition, we solve the HJB equation \eqref{eq.HJB.new}, which now takes the form
\begin{align*}
	-V^*(x)+\frac{1}{2} D^2_x V^*(x)+\ln\left(\int_{0.1}^{0.9}\exp\Big( (1-u)D_x V^* +\mathfrak u\left(u e^x\right)  \Big) du  \right)=0 \label{eq:eg.HJB},
\end{align*}
for $V^*$ using finite different methods again and compute $\pi^*$ by \eqref{pi^*}. Given all this, we compute the logarithmic values of the errors $\|V^*-v^n\|_{L^\infty(D)}$ and $\|\pi^*-\pi^n\|_{L^\infty(D\times U)}$ 
for $D=[-50, 50]$.\footnote{A bigger domain can be used, but it does not change the numerical results presented below.} Figures \ref{casei}, \ref{caseii} and \ref{caseiii} present the results, which show that $v^n\to V^*$ and $\pi^n\to\pi^*$ as $n$ increases.  It is worth noting that the two choices of the initial guess $v^0(x)$ are arbitrary. In fact, our experiments show that $v^0(x)$ can be taken (arbitrarily) to be many other functions without affecting the convergence results. That is to say, at least in this  example of optimal consumption, when all conditions of Theorem~\ref{thm.verification.limit}  are satisfied, our algorithm demonstrates certain robustness with respect to the utility function $\mathfrak u$ and the initial guess $v^0(x)$.   

Finally, let us point out that under the power utility function $\mathfrak u(y) := {y^{1-\gamma}}/{1-\gamma}$, for a fixed $1\neq\gamma>0$, the condition ``$\Lambda_2=\Lambda_{\lfloor d/2 \rfloor+2}<\infty$ for $d=1$'' fails to hold, as the reward function $r(y,u) = \mathfrak u(u e^y) = {u^{1-\gamma} e^{(1-\gamma)y}}/({1-\gamma})$ and its first and second derivatives in $y$ are all unbounded. Hence, Theorem~\ref{thm.verification.limit} cannot be applied here and it is unclear if the PIA theoretically converges. We additionally find that the above algorithm, when applied to the power utility case, is numerically unstable and does not easily compute $v^n$ and $V^*$ (let alone showing whether ``$v^n\to V^*$'' holds). It is of interest as future research to generalize Theorem~\ref{thm.verification.limit} to such a case and develop a suitable numerical scheme.

\begin{figure}[h!]
	\hspace{-0.8 cm}
	\includegraphics[width=18cm]{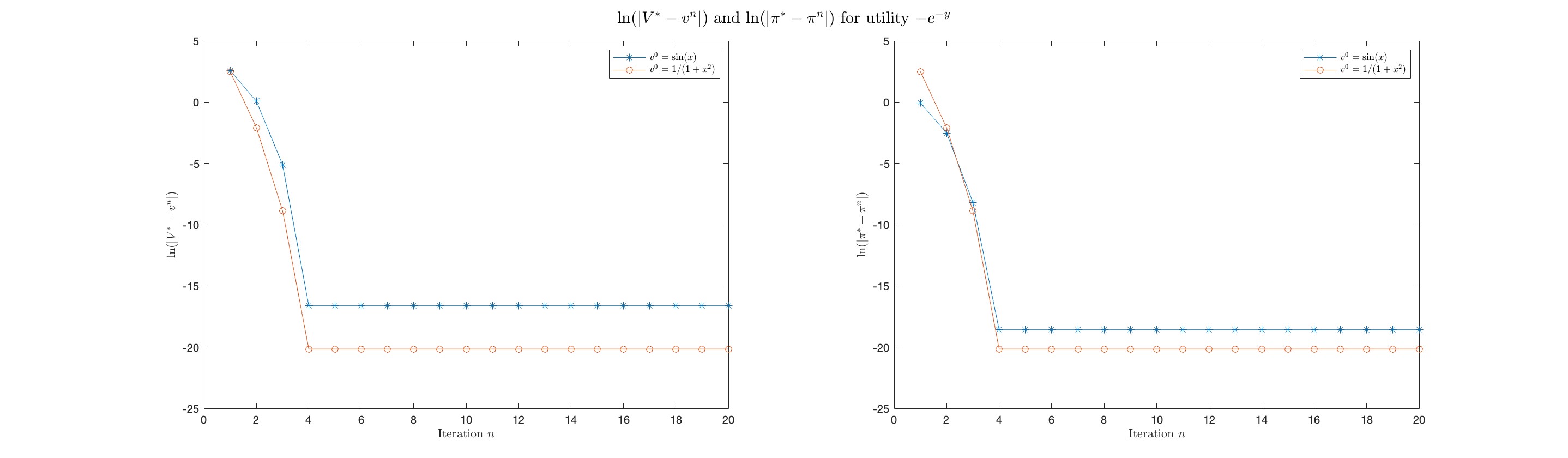}
	\caption{Difference between $V^*$ and $v^n$ (i.e., $\ln(\|V^*-v^n\|_{L^\infty([-50, 50])})$ on the left panel) and difference between $\pi^*$ and $\pi^n$ (i.e., $\ln(\|\pi^*-\pi^n\|_{L^\infty([-50, 50]\times[0.1, 0.9])}$ on the right panel) against $n=1,2,\cdots,20$, under the utility function $\mathfrak u(y)=-e^{-y}$. The blue (resp.\ orange) curve with star (resp.\ circle) markers is computed under the initial guess $v^0(x)=\sin(x)$ (resp.\ $v^0(x)=1/(1+x^2)$). 
	Notice that the algorithm is set to stop once it reaches the tolerance of the finite difference solver.}\label{casei}
	\centering
\end{figure}

\begin{figure}[h!]
	\hspace{-0.8 cm}
	\includegraphics[width=18cm]{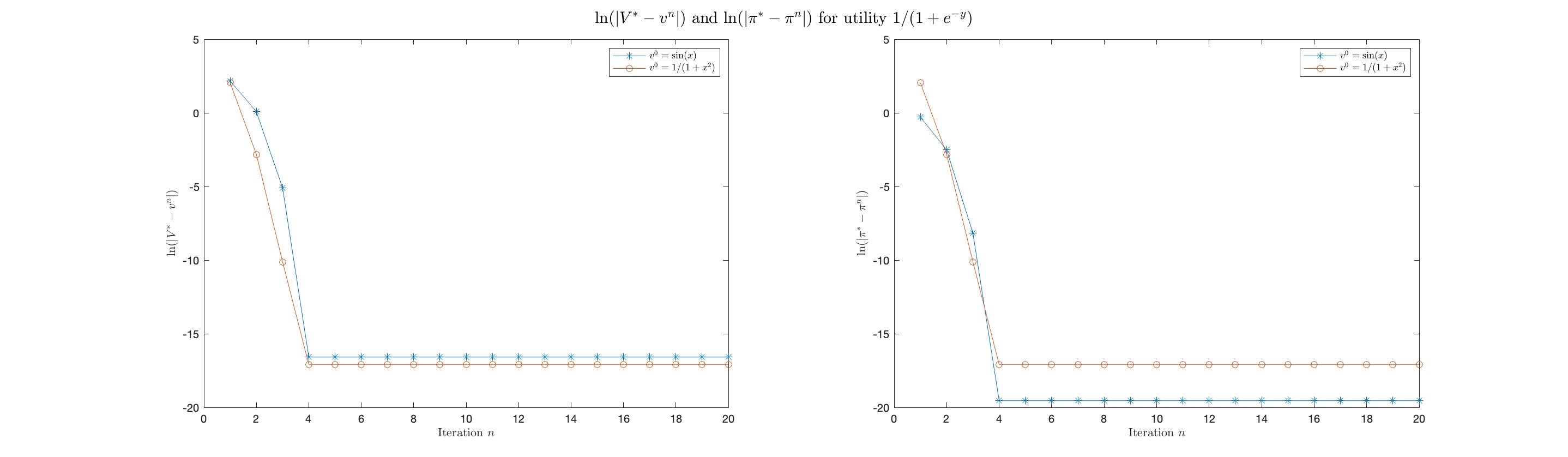}
	\caption{Difference between $V^*$ and $v^n$ (i.e., $\ln(\|V^*-v^n\|_{L^\infty([-50, 50])})$ on the left panel) and difference between $\pi^*$ and $\pi^n$ (i.e., $\ln(\|\pi^*-\pi^n\|_{L^\infty([-50, 50]\times[0.1, 0.9])}$ on the right panel) against $n=1,2,\cdots,20$, under the utility function $\mathfrak u(y)=1/(1+e^{-y})$. The blue (resp.\ orange) curve with star (resp.\ circle) markers is computed under the initial guess $v^0(x)=\sin(x)$ (resp.\ $v^0(x)=1/(1+x^2)$). 
	Notice that the algorithm is set to stop once it reaches the tolerance of the finite difference solver.
	}\label{caseii}
	\centering
	\vspace{-0.5cm}
\end{figure}

\begin{figure}[h!]
	\hspace{-0.8 cm}
	\includegraphics[width=18cm]{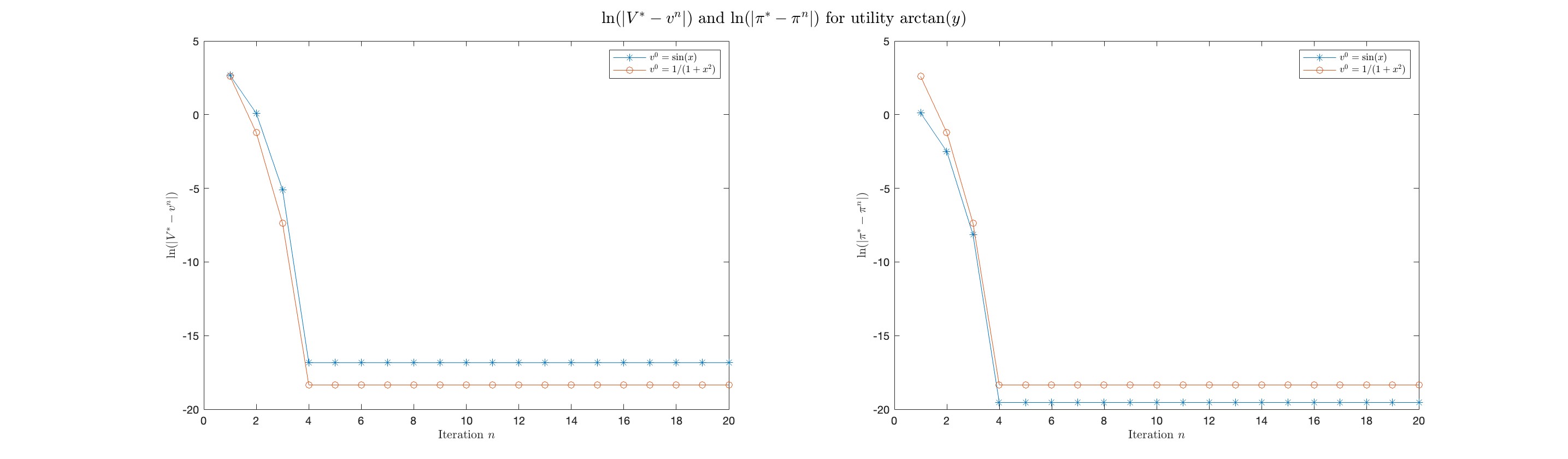}
	\caption{Difference between $V^*$ and $v^n$ (i.e., $\ln(\|V^*-v^n\|_{L^\infty([-50, 50])})$ on the left panel) and difference between $\pi^*$ and $\pi^n$ (i.e., $\ln(\|\pi^*-\pi^n\|_{L^\infty([-50, 50]\times[0.1, 0.9])}$ on the right panel) against $n=1,2,\cdots,20$, under the utility function $\mathfrak u(y)=\arctan(y)$. The blue (resp.\ orange) curve with star (resp.\ circle) markers is computed under the initial guess $v^0(x)=\sin(x)$ (resp.\ $v^0(x)=1/(1+x^2)$). 
	Notice that the algorithm is set to stop once it reaches the tolerance of the finite difference solver.
	}
	\label{caseiii}
	\centering
\end{figure}

\begin{appendix}

\section{Derivation of Lemma \ref{lemma.lips.xy}}\label{sec:appendix}

Recall \eqref{eq.parameters.m}--\eqref{eq.h}. We will prove (i) and (ii) in Lemma \ref{lemma.lips.xy} separately. 

\begin{proof}[Proof for Lemma \ref{lemma.lips.xy} (i)] 
It is sufficient to prove only \eqref{eq.lips.xy}, as the rest of the statement in part (i) directly follows. Also, we will only prove \eqref{eq.lips.xy} for 
	\be\label{f in step 2}
	f(x,y,u) = r(x,u)-\Hc(x,y,u),
	\ee
	as the case of $f=b(x,u)$ is similar and simpler.
	Let us first define 
	\begin{equation}\label{g}
		g=g(x,y,u):=(b(x,u)\cdot y+r(x,u))/{\lambda},\quad\forall x,y\in\R^d\ \hbox{and}\ u\in U.
	\end{equation}
	For any $k,m\in \N_0$, let $y\mapsto \mathfrak{p}^{k,m}(y)$ denote a generic 
	polynomial 
	of degree 
	$k$ or smaller, with coefficients being the products of the derivatives of $x\mapsto b(x,u),r(x,u)$ 
	of order $m$ or smaller. 
	It can be checked directly that for any $i,j=1,\ldots,d$ and  $k,\bar{k},m, \bar{m}\in \N_0$, 
	\be\label{eq.notation.polynomial}
	\begin{cases}
		\poly^{k, m}(y)\cdot \poly^{\bar{k},\bar{m}}(y)=\poly^{k+\bar{k}, m\vee \bar{m}}(y)=\poly^{k+\bar{k}, m+\bar{m}}(y),\\
		D_{y_i} g=\poly^{0,0}(y),\quad  D_{x_j} g=\poly^{1,1}(y)\\
		D_{y_i}\poly^{0,m}(y)=0, \quad 
		D_{y_i}\poly^{k+1,m}(y)=\poly^{k,m}(y),\quad  D_{x_j}  \poly^{k,m}(y)=\poly^{k,m+1}(y).
	\end{cases} 
	\ee 
	
	{\bf Step 1.} We will provide a generic formula for the derivatives of $(x,y)\mapsto \Gamma(x,y,u)$ in \eqref{eq.gibbs.pi}. Given $k,m\in \N_0$, for any multi-indices $a,c$ with $|a|_{l_1}=k$ and $|c|_{l_1}=m$, we claim that
	\be\label{eq.lips.pif}  
	D^c_x( D^a_y \Gamma)=\dfrac{\sum\Big( \poly^{m,m}(y) \exp(g) \big({\prod^{{2^{k+m}-1}}}\int_U \poly^{m,m}(y)\exp(g)du \big)\Big)}{(\int_U \exp(g)du)^{2^{k+m}}}.
	\ee
	Here, we do not write out the number of terms in the summation ``$\sum$" explicitly (but note that it is determined by only $k$, $m$ and $d$), while the superscript of ``$\prod$" denotes the number of terms in the product. We will use $\sum$ and $\prod$ in the same way throughout the proof. 
	
	To begin with, let us prove by induction on $k\in\N_0$ that for any multi-index $a$ with $|a|_{l_1}=k\geq 0$, 
	\be\label{eq.piy.induction} 
	D^a_y \Gamma=\dfrac{\sum \Big(\poly^{0,0}(y) \exp(g) \big(\prod^{2^k-1}\int_U \poly^{0,0}(y) \exp(g)du \big) \Big)}{(\int_U \exp(g)du)^{2^k}}. 
	\ee
	For $k=0$, \eqref{eq.piy.induction} holds as $\Gamma={\exp(g)}/{\int_U \exp(g)du}$ by definition. Suppose that \eqref{eq.piy.induction} holds for $k=n\in\N_0$ and we aim to establish \eqref{eq.piy.induction} for $k=n+1$. For any multi-index $a=(a_1,\ldots,a_d)$ with $|a|_{l_1}=n+1$, assume $a_1\geq 1$ without loss of generality (w.l.o.g.) and set $a':=(a_1-1,\ldots,a_d)$. As $|a'|_{l_1}=n$, the induction hypothesis implies $D^{a'}_y\Gamma =\texttt{(I)}/\texttt{(II)}$, with $\texttt{(I)}:=\sum \poly^{0,0}(y) \exp(g)\big(\prod^{2^n-1}\int_U \poly^{0,0}(y) \exp(g)du\big)$ and $\texttt{(II)}:=(\int_U \exp(g)du)^{2^n}$. Hence, 
	\be\label{eq.qoutient.derivterms} 
	D^a_y\Gamma =	D_{y_1}(D^{a'}_y\Gamma)=\left({\{\frac{\partial }{\partial y_1}\texttt{(I)}\}\cdot \texttt{(II)}- \{\frac{\partial }{\partial y_1}\texttt{(II)}\}\cdot \texttt{(I)}}\right)/{\texttt{(II)}^2}.
	\ee
	By a direct calculation and using \eqref{eq.notation.polynomial}, we have
	\be\label{eq.numer.induct} 
	\begin{aligned}
		\bigg\{\frac{\partial }{\partial y_1}\texttt{(I)}\bigg\}\cdot \texttt{(II)} 
		&=\bigg\{\sum \poly^{0,0}(y) \exp(g) (D_{y_1}g )\Big(\prod^{2^n-1}\int_U \poly^{0,0}(y)\exp(g)du\Big) \\
		&\qquad+ \sum \poly^{0,0}(y) \exp(g)\Big(\prod^{2^n-2}\int_U \poly^{0,0}(y)\exp(g)du\Big)\\
		&\qquad \cdot \Big( \int_U \poly^{0,0}(y)\exp(g)(D_{y_1}g ) du\Big) \bigg\}
		\cdot \texttt{(II)} \\
		&= \sum \poly^{0,0}(y) \exp(g)\Big(\prod^{2^{n+1}-1}\int_U\poly^{0,0}(y) \exp(g)du\Big),
	\end{aligned}
	\ee 
	where the last line stems from writing $\texttt{(II)}=( \int_U \exp(g)du)^{2^n}$ as $( \int_U \poly ^{0,0}(y)\exp(g)du)^{2^n}$. Similarly, 
	\be\label{eq.denom.induct}
	\begin{aligned}
		\bigg\{\frac{\partial }{\partial y_1}\texttt{(II)}\bigg\}\cdot \texttt{(I)} &=\left\{2^n \Big(\int_U \exp(g)du\Big)^{2^n-1} \Big( \int_U\exp(g) (D_{y_1}g )du \Big) \right\} \cdot \texttt{(I)}\\
		&= \sum \poly^{0,0}(y)\exp(g)\Big(\prod^{2^{n+1}-1}\int_U \poly^{0,0}(y) \exp(g)du\Big).
	\end{aligned}
	\ee
	Plugging \eqref{eq.numer.induct} and \eqref{eq.denom.induct} into \eqref{eq.qoutient.derivterms} yields \eqref{eq.piy.induction} with $k=n+1$. Thus, \eqref{eq.piy.induction} holds for all $k\in\N_0$. 
	
	We now prove \eqref{eq.lips.pif}. 
	For $m=0$ and $k\in\N_0$, \eqref{eq.lips.pif} holds as it reduces to \eqref{eq.piy.induction}. Suppose that \eqref{eq.lips.pif} holds for $m=n\in\N_0$ and $k\in \N_0$, and we aim to prove \eqref{eq.lips.pif} for $m=n+1$ and $k\in \N_0$. For any multi-indices $a=(a_1,\ldots,a_d)$ and  $c=(c_1,\ldots,c_d)$ with $|a|_{l_1}=k$ and $|c|_{l_1}=n+1$, assume $c_1\geq 1$ w.l.o.g. and set $c':=(c_1-1,\ldots,c_d)$. As $|c'|_{l_1}=n$, the induction hypothesis gives $D^{c'}_x(D^a_y\Gamma) =\texttt{(I)}/\texttt{(II)}$ with $\texttt{(I)}:=\sum \big(\poly^{n,n}(y) \exp(g) \big(\prod^{2^{k+n}-1}\int_U \poly^{n,n}(y)\exp(g)du\big)\big)$ and $\texttt{(II)}:=(\int_U \exp(g)du)^{2^{k+n}}$. Thus, 
	\be\label{eq.qoutient.derivtermsx} 
	D^c_x(D^a_y\Gamma) = D_{x_1}(D^{c'}_xD^a_y\Gamma)=\frac{(\frac{\partial }{\partial x_1}\texttt{(I)})\cdot \texttt{(II)}- (\frac{\partial }{\partial x_1}\texttt{(II)})\cdot \texttt{(I)}}{\text{\texttt{(II)}}^2}.
	\ee
	A direct calculation shows that $\frac{\partial }{\partial x_1}\texttt{(I)}$ is equal to
	\bee
	\begin{aligned}
		&\sum \bigg[ \Big(D_{x_1} \poly^{n,n}(y)+\poly^{n,n}(y)(D_{x_1}g)\Big)\exp(g)\Big(\prod^{2^{k+n}-1}\int_U \poly^{n,n}(y)\exp(g)du\Big)  \\
		&\hspace{0.35in} +  \Big(\poly^{n,n}(y) \exp(g)\Big)\Big( \prod^{2^{k+n}-2}\int_U \poly^{n,n}(y)\exp(g)du \Big)
\cdot \int_U \Big(D_{x_1}\poly^{n,n}(y) +\poly^{n,n}(y)(D_{x_1}g)\Big)\exp(g) du \bigg].
	\end{aligned}
	\eee
	By \eqref{eq.notation.polynomial}, $D_{x_1} \poly^{n,n}(y) = \poly^{n,n+1}(y)$ and  $\poly^{n,n}(y)(D_{x_1}g) =\poly^{n+1,n+1}(y)$. Plugging this into the above gives 
	\be\label{eq.numer.inductx}
	\begin{aligned}
		&\bigg\{\frac{\partial }{\partial x_1}\texttt{(I)}\bigg\}\cdot \texttt{(II)}
		= 	\sum \Big(\poly^{n+1,n+1}(y) \exp(g)\Big)\cdot \Big(\prod^{2^{k+n+1}-1}\int_U \poly^{n+1,n+1}(y)\exp(g)du\Big).
	\end{aligned}
	\ee
	By using \eqref{eq.notation.polynomial} in a similar fashion, we get 
	\be\label{eq.denom.inductx}
	\begin{aligned}
		\bigg\{\frac{\partial }{\partial x_1}\texttt{(II)}\bigg\}\cdot \texttt{(I)} 
= &\bigg\{2^{k+n}\bigg(\int_U \exp(g)du\bigg)^{2^{k+n}-1}\int_U (D_{x_1}g) \exp(g)du\bigg\} \cdot \texttt{(I)}\\
		=&\sum \Big(\poly^{n+1,n+1}(y) \exp(g)\Big)\cdot \Big(\prod^{2^{k+n+1}-1}\int_U \poly^{n+1,n+1}(y)\exp(g)du\Big).
	\end{aligned}
	\ee
	Plugging \eqref{eq.numer.inductx} and \eqref{eq.denom.inductx} into \eqref{eq.qoutient.derivtermsx} yields \eqref{eq.lips.pif}  with $m=n+1$ and $k\in\N_0$. We then conclude that \eqref{eq.lips.pif} holds for any $m,k\in\N_0$. 
	
	{\bf Step 2.} We will provide a generic formula for the derivatives of $(x,y)\mapsto f(x,y,u)$ in \eqref{f in step 2}. 
	Given $k,m\in \N_0$ with $k+m\geq 1$, for any multi-indices $a,c$ with $|a|_{l_1}=k$ and $|c|_{l_1}=m$, we claim that
	\be\label{eq.lips.pif'}  
	D^a_y(D^c_x f)=\poly^{1,m}(y)+\dfrac{\sum ({\prod^{2^{k+m-1}}} \int_U \exp(g)\poly^{m,m}(y)du)}{(\int_U\exp(g)du)^{2^{k+m-1}}}.
	\ee
	
	First, let us prove \eqref{eq.lips.pif'} for $k=0$, i.e., for any multi-index $c=(c_1,\ldots,c_d)$ with $|c|_{l_1} =m\in\N$, 
	\be\label{step 2 induction}
	D^c_x f=\poly^{1,m}(y)+\dfrac{\sum ({\prod^{2^{m-1}}} \int_U \exp(g)\poly^{m,m}(y)du)}{(\int_U\exp(g)du)^{2^{m-1}}},
	\ee
	If $m=1$, as $\Hc = \lambda \ln\Gamma= \lambda(g-\ln(\int_U\exp(g)du))$ by \eqref{eq.h} and \eqref{g},  we have
	\be\label{m=1} 
	\begin{aligned}
		D_{x_j} f=& D_{x_j} r+\lambda D_{x_j} g-\lambda\frac{\int_U (D_{x_j}g) \exp(g)du}{\int_U \exp(g) du}
		=\poly^{1,1}(y)+\frac{\int_U \poly^{1,1}(y) \exp(g)du}{\int_U \exp(g) du}\quad \forall 1\leq j\leq d. 
	\end{aligned}
	\ee 
	That is, \eqref{step 2 induction} holds for $m=1$. Suppose that \eqref{step 2 induction} holds for $m=n\in\N$ and we aim to prove \eqref{step 2 induction} for $m=n+1$. For any multi-index $c=(c_1,\ldots,c_d)$ with $|c|_{l_1}=n+1$,  assume $c_1\geq 1$ w.l.o.g.\ and set $c':=(c_1-1,\ldots,c_d)$. As $|c'|_{l_1}=n$, $D^{c'}_x f$ fulfills \eqref{step 2 induction} (by the induction hypothesis), so that
	\be\label{eq.qoutient.derivtermsxf} 
	D_{x_1}(D^{c'}_xf)=D_{x_1}(\poly^{1,n}(y))+ \dfrac{(\frac{\partial }{\partial x_1}\texttt{(I)})\cdot \texttt{(II)}- (\frac{\partial }{\partial x_1}\texttt{(II)})\cdot \texttt{(I)}}{\texttt{(II)}^2},
	\ee
	with $\texttt{(I)}:=\sum (\prod^{2^{n-1}} \int_U \exp(g)\poly^{n,n}(y)du)$ and $\texttt{(II)}:=(\int_U\exp(g)du)^{2^{n-1}}$. Using \eqref{eq.notation.polynomial} in the same way as in the derivation of \eqref{eq.numer.inductx} and \eqref{eq.denom.inductx}, we see that both $(\frac{\partial }{\partial x_1}\texttt{(I)})\cdot \texttt{(II)}$ and $(\frac{\partial }{\partial x_1}\texttt{(II)})\cdot \texttt{(I)}$ can be expressed as $\sum \big(\prod^{2^{n}} \int_U \exp(g)\poly^{n+1,n+1}(y)du \big).$
	This, together with \eqref{eq.qoutient.derivtermsxf}, shows that \eqref{step 2 induction} holds for $m=n+1$. We then conclude that \eqref{step 2 induction} holds for any $m\in\N$.
	
	We proceed to prove \eqref{eq.lips.pif'} for $k\in\N$. If $k=1$ and $m=0$, by a calculation similar to \eqref{m=1}, 
	$$
	\begin{aligned}
		D_{y_i} f  &=\lambda D_{y_i} g-\lambda\frac{\int_U (D_{y_i}g) \exp(g)du}{\int_U \exp(g) du}=\poly^{0,0}(y)+\frac{\int_U\poly^{0,0}(y) \exp(g)du}{\int_U \exp(g) du}\quad \forall 1\leq i\leq d,
	\end{aligned}
	$$
	i.e., \eqref{eq.lips.pif'} holds in this case. 
	Based on this and \eqref{step 2 induction}, 
	we suppose that \eqref{eq.lips.pif'} holds for $k=n\in\N_0$ and $m\in \N_0$, with $n+m\ge 1$, and aim to prove \eqref{eq.lips.pif'} for $k=n+1$ and $m\in \N_0$. For any multi-indices $a=(a_1,\ldots,a_d)$ and $c=(c_1,\ldots,c_d)$ with $|a|_{l_1}=n+1$ and  $|c|_{l_1}=m$, assume $a_1\geq 1$ w.l.o.g.\ and set $a':=(a_1-1,\ldots,a_d)$. As $|a'|_{l_1}=n$, $D^{a'}_yD^{c}_xf$ fulfills \eqref{eq.lips.pif'} (by the induction hypothesis), so that
	\be\label{eq.qoutient.derivtermsyf} 
	D_{y_1}(D^{a'}_yD^{c}_xf)=D_{y_1}(\poly^{1,m}(y))+ \frac{(\frac{\partial }{\partial y_1}\texttt{(I)})\cdot \texttt{(II)}- (\frac{\partial }{\partial y_1}\texttt{(II)})\cdot \texttt{(I)}}{\texttt{(II)}^2},
	\ee
	with $\texttt{(I)} :=\sum (\prod^{2^{n+m-1}} \int_U \exp(g)\poly^{m,m}(y)du)$ and $\texttt{(II)}:=(\int_U\exp(g)du)^{2^{n+m-1}}$. Using \eqref{eq.notation.polynomial} in the same way as in \eqref{eq.numer.inductx}-\eqref{eq.denom.inductx}, we see that both $(\frac{\partial }{\partial y_1}\texttt{(I)})\cdot \texttt{(II)}$ and $(\frac{\partial }{\partial y_1}\texttt{(II)})\cdot \texttt{(I)}$ can be expressed as 
	$\sum \big(\prod^{2^{m+n}} \int_U \exp(g)\poly^{m+1,m+1}(y)du\big).$
	This, together with \eqref{eq.qoutient.derivtermsyf}, shows that \eqref{eq.lips.pif'} holds for $k=n+1$ and $m\in\N_0$. We then conclude that \eqref{eq.lips.pif'} holds for any $k,m\in\N_0$ with $k+m\ge 1$.
	
	{\bf Step 3.} We will provide an estimate for the derivatives of $\hat f(x,y):=\int_U f(x,y,u)\Gamma(x,y,u) du$, with $f$ given by \eqref{f in step 2}. For two multi-indices $b^{(j)}=(b^{(j)}_1,\ldots,b^{(j)}_d)$, $j=1,2$, we will write $b^{(1)}\leq b^{(2)}$ if $b^{(1)}_i\leq b^{(2)}_i$ for all $i=1,2,\ldots,d$. In the sequel, $C_{m}>0$  (resp.\ $C_{k,m}>0$) is a generic constant that depends on only $m$,  $\Lambda_m$,  $\lambda$, $d$ and $\Leb(U)$ (resp.\ $k$, $m$, $\Lambda_m$, $\lambda$, $d$, and $\Leb(U)$) and may vary from line to line. Given $k,m\in \N_0$, for any multi-indices $a,c$ with $|a|_{l_1}=k$ and $|c|_{l_1}=m$, we claim that
	\be\label{eq.f.y}  
	\begin{aligned}
		\left|D^a_y D^c_x \hat f \right| = \bigg|D^a_y D^c_x \int_U f\Gamma du\bigg|\leq C_{k+m} (1+|y|)^{(m+1)2^{k+m}},
	\end{aligned}
	\ee
	For the case $k=m=0$, a calculation similar to \eqref{estimate in step 3} shows 
	\begin{equation}\label{a=c=0}
		|f| \le 3\Lambda_0(1+|y|)+ \lambda |\ln(\Leb(U))|,
	\end{equation}
	which implies $|\int_U f\Gamma du|\leq C_{0,0} (1+|y|)$. That is, \eqref{eq.f.y} holds for $k=m=0$. 
	
	In \eqref{eq.lips.pif} and \eqref{eq.lips.pif'}, every $\poly^{m,m}(y)$ is a polynomial in $y$ of degree $m$ or smaller, with coefficients being products of derivatives of $x\mapsto b(x,u), r(x,u)$ of order $m$ or smaller; moreover, the number of terms contained in each product depends on how many times differentiation has been carried out, i.e., $m$. Hence, in \eqref{eq.lips.pif} and \eqref{eq.lips.pif'}, we have $|\poly^{m,m}(y)|\leq C_m (1+|y|^m)$. 
	Now, for any $k,m\in \N_0$ and multi-indices $a,c$ with $|a|_{l_1}=k$ and $|c|_{l_1}=m$, applying the above estimate to \eqref{eq.lips.pif} yields
	\be\label{eq.f.0}  
	\begin{aligned}
		|D^c_x (D^a_y   \Gamma)| &\leq  \dfrac{ \sum \big(C_{m}(1+|y|^m)\big)^{2^{k+m}}\exp(g) \big(\int_U \exp(g)du\big)^{2^{k+m}-1}}{\big(\int_U \exp(g)du\big)^{2^{k+m}}}\\
		&=   \sum C_{k,m}(1+|y|)^{m2^{k+m}}\Gamma
		= C_{k,m}(1+|y|)^{m2^{k+m}}\Gamma.
	\end{aligned}
	\ee	
	Similarly, when $k+m\ge 1$, applying the same estimate to \eqref{eq.lips.pif'} gives 
	\begin{align*}
		|D^a_y(D^c_x f)|
		&\leq  |\poly^{1,m}(y)|+C_{k,m}(1+|y|)^{m2^{k+m-1}}\frac{\sum (\int_U \exp(g)du)^{2^{k+m-1}}}{(\int_U\exp(g)du)^{2^{k+m-1}}}
		\leq  C_{k,m}(1+|y|)^{(m+1) 2^{k+m-1}}. 
	\end{align*}
	Combining this and \eqref{a=c=0}, which covers the case $k+m=0$, leads to 
	\be\label{eq.f.1}
	|D^a_y(D^c_x f)|\le C_{k,m}(1+|y|)^{(m+1) 2^{k+m}}.
	\ee	
	
	Finally, given $k,m\in \N_0$ with $k+m\geq 1$, for any multi-indices $a,c$ with $|a|_{l_1}=k$ and $|c|_{l_1}=m$,	
	\be 
	\begin{aligned}
		&
		\bigg|D^a_y D^c_x \int_U f\Gamma du\bigg|=\bigg|\int_U D^a_y D^c_x(f\Gamma) du \bigg|
		 \leq \sum_{0\leqslant p\leqslant a,0\leqslant q \leqslant c} \int_U |(D^{p}_yD^q_x f) (D^{a-p}_y D^{c-q}_x \Gamma)|du\\
		&\leq  \sum_{0\leqslant p\leqslant a,0\leqslant q \leqslant c}C_{|p|_{l_1},|q|_{l_1}}  C_{|a-p|_{l_1},|c-q|_{l_1}} 
\cdot (1+|y|)^{(|q|_{l_1}+1)2^{(|p|_{l_1}+|q|_{l_1})}+(m-|q|_{l_1})2^{(k-|p|_{l_1}+m-|q|_{l_1})}}\\
		&\le C_{k,m} (1+|y|)^{(m+1)2^{k+m}}
		\leq  C_{k+m,k+m} (1+|y|)^{(m+1)2^{k+m}}
		= C_{k+m} (1+|y|)^{(m+1)2^{k+m}},
	\end{aligned}
	\ee
	where the second inequality follows from \eqref{eq.f.0}, \eqref{eq.f.1}, and $\int_U \Gamma du=1$. That is, \eqref{eq.f.y} holds for $k,m\in\N_0$ with $k+m\ge 1$. 
	
	{\bf Step 4.} We now prove \eqref{eq.lips.xy}.  
	Fix $k\in\N_0$ and $p=(p_i)_{i=1}^d:\R^d\to \R^d$ in $\Cc^{k,\alpha}(\R^d)$. For $k=0$, we recover \eqref{eq.lips.xy} from \eqref{eq.f.y} by taking $k=m=0$ and $y=p(x)$ therein. For $k\in\N$, we deduce from \eqref{eq.f.y} that for any multi-index $a=(a_1,\ldots,a_d)$ with $|a|_{l_1}=k$, 
	\be\label{eq.f.z0}
	\big|D^a_z\hat{f}(z)\big|\leq C_{k} (1+|y|)^{(k+1)2^{k}},\quad \hbox{with $z=(x,y)\in\R^{d}\times\R^d$}. 
	\ee
	For any $s\in\N$, we say that the multi-index $a$ is decomposed into distinct $q_1, q_2,\ldots,q_s\in\N_0^d$ with multiplicities $m_1,m_2,\ldots,m_s\in\N_0^{2d}$ if $a = \sum_{j=1}^s |m_j|_{l_1}q_j$. Note that $q_j$ (resp.\ $m_j$) is a multi-index of dimension $d$ (resp.\ $2d$) for all $1\le j\le s$. We denote by $\mathcal{D}$ the set of all $(s,q=(q_j)_{j=1}^s, m=(m_j)_{j=1}^s)$ that form a decomposition of $a$. Now, for all $x=(x_1,\ldots,x_d)\in\R^d$ define $\fz:\R^d\to\R^{2d}$ by 
	$$\fz(x) = (\fz_1(x),\ldots,\fz_d(x), \fz_{d+1}(x),\ldots, \fz_{2d}(x)):= (x_1,\ldots,x_d,p_1(x),\ldots,p_d(x)).$$
	Then, by the multivariate chain rule (i.e., the Fa\`a di Bruno formula; see \cite[Section 6]{ma2009higher}), 
	\be\label{eq.chain.multi}  
	D^a_x \hat{f}(x,p(x))=a! \sum_{(s,q,m)\in \mathcal{D}} (D^{\bar m}_z \hat{f})(\fz(x)) \prod_{j=1}^s \prod_{i=1}^{2d}\frac{1}{m_{ji}!}\Big[\frac{1}{q_{j1}!\cdots q_{jd}!}D^{q_j}_x \fz_i(x) \Big]^{m_{ji}}, 
	\ee
	with $a! := \prod_{j=1}^d a_j!$ and $\bar m := \sum_{j=1}^s m_j$. By the definition of $\fz_i$, 
	for any multi-index $c=(c_1,\ldots,c_d)$, 
	\be\label{Dfz}
	D^c_x \fz_i(x)=1\ \hbox{or}\ 0\quad \hbox{for $1\leq i\leq d$};\qquad D^c_x \fz_i(x)=D^c_x p_{i-d}(x)\quad \hbox{for $d+1\le i\leq 2d$}. 
	\ee
	It follows that
	\be\label{eq.deriv.z}   
	\begin{aligned}
		\bigg|\prod_{j=1}^s \prod_{i=1}^{2d}\Big[D^{q_j}_x \fz_i \Big]^{m_{ji}}\bigg| &\leq  \max\bigg\{1, \bigg|\prod_{j=1}^s \prod_{i=d+1}^{2d}\Big[D^{q_j}_x \fz_i \Big]^{m_{ji}}\bigg| \bigg\}
		\leq 1+ \prod_{j=1}^s \prod_{i=d+1}^{2d}\bigg( \sum_{|c|_{l_1}=0,\ldots,k} |D^c_x p|\bigg)^{m_{ji}}\\
		&= 1+\bigg(\sum_{|c|_{l_1}=0,\ldots,k} |D^c_x p|\bigg)^{\sum_{j=1}^s |m_j|_{l_1}}
		\leq   1+\bigg(\sum_{|c|_{l_1}=0,\ldots,k} |D^c_x p|\bigg)^k,
	\end{aligned}
	\ee 
	where the second and last inequalities stem from $(\max_{1\leq j\leq s}|q_j|_{l_1})\vee(\sum_{j=1}^s |m_j|_{l_1}) \leq |a|_{l_1}= k$, a consequence of $(s,q,m)\in\mathcal D$. Also, collecting all the constant coefficients in \eqref{eq.chain.multi} results in 
	$a! \sum_{(s,q,m)\in \mathcal{D}}\ \prod_{j=1}^s \prod_{i=1}^{2d}\frac{1}{m_{j,i}!}(\frac{1}{q_{j,1}!\cdots q_{j,d}!} )^{m_{j,i}}$, which is a finite constant depending on only $k$ and $d$. Hence, by applying the estimates \eqref{eq.f.z0} and \eqref{eq.deriv.z} to \eqref{eq.chain.multi}, we get the desired result \eqref{eq.lips.xy}. 
\end{proof}

Now we move to the proof for part (ii) of Lemma \ref{lemma.lips.xy}.

\begin{proof}[Proof of Lemma~\ref{lemma.lips.xy} (ii)]
We will again prove the result only for $(\hat r-\hat\Hc)(\cdot,p(\cdot))$, as the case of $\hat b(\cdot,p(\cdot))$ is similar and simpler. 
	Fix $k\in \N_0$ and $0<\alpha\leq 1$. Consider $f$ as in \eqref{f in step 2}. To prove the desired result $\hat f(\cdot,p(\cdot))\in \Cc^{k,\alpha}_{\text{unif}}(\R^d)$, it suffices to show $D^a_x\hat{f}(\cdot,p(\cdot))\in \Cc^{0,\alpha}_{\text{unif}}(\R^d)$ for all multi-indices $a=(a_1,\ldots,a_d)$ with $|a|_{l_1}\le  k$. 
	First, it can be checked directly that the following implication holds:
	\be\label{h1h2}
	h_1, h_2\in \Cc^{0,\alpha}_{\text{unif}}(\R^d)\implies h_1+h_2,\ h_1h_2\in \Cc^{0,\alpha}_{\text{unif}}(\R^{d}).
	\ee
	By \eqref{eq.chain.multi}, the above implication indicates that ``$D^a_x\hat{f}(\cdot,p(\cdot))\in \Cc^{0,\alpha}_{\text{unif}}(\R^d)$'' holds if $(D^{\bar m}_z \hat{f})(\fz(\cdot))$ and $D^{q_j}_x \fz_i$ in \eqref{eq.chain.multi} belong to $\Cc^{0,\alpha}_{\text{unif}}(\R^d)$. By \eqref{Dfz} and 
	$p\in\Cc^{k,\alpha}_{\text{unif}}(\R^{d})$, we know $D^{q_j}_x \fz_i\in\Cc^{0,\alpha}_{\text{unif}}(\R^d)$. Hence, it remains to show $(D^{\bar m}_z \hat{f})(\fz(\cdot))\in \Cc^{0,\alpha}_{\text{unif}}(\R^d)$, or equivalently, $(D^{a}_z \hat{f})(\cdot, p(\cdot))\in \Cc^{0,\alpha}_{\text{unif}}(\R^d)$ for any multi-index $a=(a_1,\ldots,a_{2d})$ with $|a|_{l_1}\le k$. It follows from  \eqref{eq.hat.notation} that 
for all $a\in\N_0^{2d}$ with $|a|_{l_1}\le k$,
	\[
	\big\|(D^a_z \hat f)(\cdot,p(\cdot))\big\|_{\Cc^{0,\alpha}(B_1(x))} \le \int_U \big\|(D^a_z (f\Gamma))(\cdot,p(\cdot),u)\big\|_{\Cc^{0,\alpha}(B_1(x))} du. 
	\]
	As $\Leb(U)<\infty$, for ``$(D^{a}_z \hat{f})(\cdot, p(\cdot))\in \Cc^{0,\alpha}_{\text{unif}}(\R^d)$ for $a\in\N_0^{2d}$ with $|a|_{l_1}\le k$'' to hold, it suffices to show
``$
	\sup_{x\in\R^d, u\in U}\big\|D^a_z (f\Gamma)(\cdot,p(\cdot),u)\big\|_{\Cc^{0,\alpha}(B_1(x))}<\infty\ \hbox{for}\ a\in\N_0^{2d}\ \hbox{with}\ |a|_{l_1}\le k,
$''
	which is equivalent to: for any $a_1,c_1,a_2,c_2\in \N_0^d$ such that $0\le |a_1|_{l_1} + |c_1|_{l_1} + |a_2|_{l_1} + |c_2|_{l_1}\le k$, 
	\be\label{step 5 to show}
	\sup_{x\in\R^d, u\in U}\big\| \big(D^{a_1}_x (D^{c_1}_y f) D^{a_2}_x (D^{c_2}_y \Gamma)\big) (\cdot,p(\cdot),u)\big\|_{\Cc^{0,\alpha}(B_1(x))}<\infty. 
	\ee
	To establish this, we first observe similarly to \eqref{h1h2} that, for any $h_1, h_2:\R^d\times U\rightarrow \R$, 
	\be\label{eq.alpha.unif}
	h_1, h_2\in \Cc^{0,\alpha}_{\text{unif},U}(\R^d)\implies h_1+h_2,\ h_1h_2\in \Cc^{0,\alpha}_{\text{unif},U}(\R^{d}),
	\ee
	where we denote by $\Cc^{0,\alpha}_{\text{unif},U}(\R^d)$ the set of $h:\R^d\times U\to \R$ with $\sup_{x\in\R^d, u\in U}\|h(\cdot,u)\|_{\Cc^{0,\alpha}(B_1(x))}<\infty$. Thanks to \eqref{eq.alpha.unif}, we deduce from \eqref{eq.lips.pif} and \eqref{eq.lips.pif'} that to establish \eqref{step 5 to show},  it suffices to show $\poly^{m,m}, \exp(g), 1/\exp(g)\in \Cc^{0,\alpha}_{\text{unif},U}(\R^d)$ in \eqref{eq.lips.pif} and \eqref{eq.lips.pif'}, with $y=p(x)$ and $0\le m\le k$ therein. 
	
	With $y=p(x)$ and $0\le m\le k$, note that we can express every $\poly^{m,m}(y)$ in \eqref{eq.lips.pif} and \eqref{eq.lips.pif'} more specifically as a function on $\R^d\times U$, i.e., 
	\be\label{eq.poly.fullform} 
	\poly^{m,m}(x,u) = \poly^{m,m}(x,p(x),u)=\sum_{i} c_i(x,u)  (p_1(x))^{a_{i1}}\cdots (p_d(x))^{a_{id}} ,
	\ee
	where each $c_i(x,u)$ is a product of the derivatives of $x\mapsto b(x,u), r(x,u)$ of order $m$ or smaller, and each $a_i=(a_{i1},\ldots, a_{id})$ is a multi-index with $|a_i|_{l_1}\leq m$. 
Observe from \eqref{eq.parameters.m}  that $\sup_{u\in U}\{\|b(\cdot, u)\|_{\Cc^{k,\alpha}(B_1(x))}+\|r(\cdot, u)\|_{\Cc^{k,\alpha}(B_1(x))}\}\le 2^{1-\alpha}\Lambda_{k+1}<\infty$ for all $x\in\R^d$, which implies $b,r\in \Cc^{k,\alpha}_{\text{unif},U}(\R^d)$. By \eqref{eq.alpha.unif}, we get $c_i\in \Cc^{0,\alpha}_{\text{unif},U}(\R^d)$. As $p\in \Cc^{k,\alpha}_{\text{unif}}(\R^d)$, each $p_i$, when viewed as a function on $\R^d\times U$ (constant on $U$), readily lies in $\Cc^{0,\alpha}_{\text{unif},U}(\R^d)$. Hence, in view of \eqref{eq.poly.fullform}, we deduce from \eqref{eq.alpha.unif} again that $\poly^{m,m}\in \Cc^{0,\alpha}_{\text{unif},U}(\R^d)$. 
	
	Set $\bar g(x,u) := g(x,p(x),u)$. It remains to show $\exp(\bar g), 1/\exp(\bar g)\in \Cc^{0,\alpha}_{\text{unif},U}(\R^d)$. As $b,r, p\in \Cc^{0,\alpha}_{\text{unif},U}(\R^d)$, by the definition of $g$ in \eqref{g} and using \eqref{eq.alpha.unif}, we have $\bar g \in \Cc^{0,\alpha}_{\text{unif},U}(\R^d)$. This readily gives the boundedness of $\exp(\bar g)$ and $1/\exp(\bar g)$. Set $M:= \|\exp(\bar g)\|_{L^\infty(\R^d\times U)}<\infty$. By writing $\bar g_i = \bar g(x_i,u)$, $i=1,2$, for any distinct $x_1,x_2\in\R^d$ and $u\in U$, it holds for any $x\in\R^d$ that  
	\begin{align*}
		&\hspace{-0.35in}\sup_{x_1,x_2\in B_1(x)} \frac{|\exp(\bar g_1)-\exp(\bar g_2)|}{|x_1-x_2|^\alpha} = \sup_{x_1,x_2\in B_1(x)} \frac{|\exp(\bar g_1)-\exp(\bar g_2)|}{|\bar g_1-\bar g_2|} \frac{|\bar g_1- \bar g_2|}{|x_1-x_2|^\alpha}\\
		&\le \bigg(\sup_{x\in B_1(x)} \exp(\bar g(x,u))\bigg) \|\bar g(\cdot,u)\|_{\Cc^{0,\alpha}(B_1(x))}\le M \|\bar g(\cdot,u)\|_{\Cc^{0,\alpha}(B_1(x))}.
	\end{align*}
	As $\bar g \in \Cc^{0,\alpha}_{\text{unif},U}(\R^d)$, taking supremum over $x\in\R^d$ and $u\in U$ yields a finite upper bound, which implies $\exp(\bar g) \in \Cc^{0,\alpha}_{\text{unif},U}(\R^d)$. Similarly, it holds for any $x\in\R^d$ and $u\in U$ that
	\begin{align*}
		\sup_{x_1,x_2\in B_1(x)}\frac{|1/\exp(\bar g_1)-1/\exp(\bar g_2)|}{|x_1-x_2|^\alpha} 
		&= \sup_{x_1,x_2\in B_1(x)} \frac{1}{\exp(\bar g_1)\exp(\bar g_2)}\frac{|\exp(\bar g_1)-\exp(\bar g_2)|}{|\bar g_1-\bar g_2|} \frac{|\bar g_1-\bar g_2|}{|x_1-x_2|^\alpha}\\
		&\le N^2 M \|\bar g(\cdot,u)\|_{\Cc^{0,\alpha}(B_1(x))},
	\end{align*}
with $N:= \|1/\exp(\bar g)\|_{L^\infty(\R^d\times U)}<\infty$. Taking supremum over $x\in\R^d$ and $u\in U$ yields a finite upper bound, thanks again to $\bar g \in \Cc^{0,\alpha}_{\text{unif},U}(\R^d)$. This implies $1/\exp(\bar g) \in \Cc^{0,\alpha}_{\text{unif},U}(\R^d)$. 
\end{proof}
\end{appendix}

\small{
\bibliographystyle{plain}
\bibliography{reference}
}
\end{document}